\documentclass[12pt]{l4dc2022}

\usepackage[OT1]{fontenc}
\usepackage{amsmath,amsfonts,amssymb}
\usepackage{mathtools, enumitem}
\usepackage{latexsym}
\usepackage{relsize}
\usepackage{fullpage}
\usepackage{times, natbib}
\usepackage[autostyle, english = american]{csquotes}

\title[{Robust Online Control with Model Misspecification}]{Robust Online Control with Model Misspecification}

 \author{\Name{Xinyi Chen}$^{1,2}$ \Email{xinyic@princeton.edu}\\
  \Name{Udaya Ghai}$^{1,2}$ \Email{ughai@cs.princeton.edu}\\
  \Name{Elad Hazan}$^{1,2}$ \Email{ehazan@cs.princeton.edu}\\
  \Name{Alexandre Megretski}$^{3}$ \Email{ameg@mit.edu}\\
  \addr $1$ Department of Computer Science, Princeton University \\
   $2$ Google AI Princeton\\
  $3$  Department of Electrical Engineering and Computer Science, Massachusetts Institute of Technology
  }

\definecolor{chestnut}{rgb}{0.8, 0.36, 0.36}

\newcommand{\ltgain}{\mathrm{\ell_2}\text{-gain}}

\def\R{{\mathbb R}}

\newcommand{\A}{\mathcal{A}}

\newcommand{\K}{\ensuremath{\mathcal K}}

\newcommand{\tstart}[1]{s_{#1}}
\newcommand{\tend}[1]{e_{#1}}

\DeclareMathOperator{\OPT}{OPT}
\DeclareMathOperator{\CR}{\mathcal{C}}

\newcommand{\ignore}[1]{}
\newcommand{\eh}[1]{}
\newcommand{\ug}[1]{}
\newcommand{\xc}[1]{}
\newcommand{\email}[1]{\texttt{#1}}

\newtheorem{assumption}{Assumption}

\DeclareMathAlphabet{\mathbfsf}{\encodingdefault}{\sfdefault}{bx}{n}

\DeclareMathOperator*{\argmin}{arg\,min}

\renewcommand{\O}{O}

\newcommand{\reals}{\mathbb{R}}
\newcommand{\eps}{\varepsilon}

\renewcommand{\leq}{~\le~}
\renewcommand{\geq}{~\ge~}

\let\oldtfrac\tfrac
\renewcommand{\tfrac}[2]{\smash{\oldtfrac{#1}{#2}}}

\let\nablaold\nabla
\renewcommand{\nabla}{\nablaold\mkern-2.5mu}

\newcommand{\mK}{\mathcal{K}}

\begin{document}

\maketitle

\begin{abstract}

We study online control of an unknown nonlinear dynamical system that is approximated by a time-invariant linear system with  model misspecification. Our study focuses on  {\bf robustness}, a measure of how much deviation from the assumed linear approximation can be tolerated by a controller while maintaining finite $\ltgain$. 

A basic methodology to analyze robustness is via the small gain theorem.  However, as an implication of recent lower bounds on adaptive control, this method can only yield robustness  that is exponentially small in the dimension of the system and its parametric uncertainty. The work of \cite{CusumanoPoollaACC1988} shows that much better robustness can be obtained, but the control algorithm is inefficient, taking exponential time in the worst case.

In this paper we investigate whether there exists an efficient algorithm with provable robustness beyond the small gain theorem. We demonstrate that for a fully actuated system, this is indeed attainable. We give an efficient controller that can tolerate robustness that is polynomial in the dimension and independent of the parametric uncertainty; furthermore, the controller obtains an $\ltgain$ whose dimension dependence is near optimal. 
\end{abstract}

\section{Introduction}

The problem of {linear} control of linear dynamical systems is well studied and understood. 
Classical algorithms such as 
$\mathcal{H}_2$ optimization (which includes
LQR and LQG) are known to be optimal in appropriate stochastic and worst case
settings, while robust $\mathcal{H}_\infty$ control is optimal in the worst case, assuming quadratic costs. Even though these results can be generalized to nonlinear
systems, the resulting optimal control synthesis requires solving partial differential
equations in high dimensional domains, usually an intractable task.
Beyond classical control methods, recent advancements in the machine learning community gave rise to efficient online control methods based on convex relaxations that minimize regret in the presence of  adversarial perturbations.  

In this paper we revisit a natural and well-studied approach of nonlinear control, where the nonlinear system is approximated by a linear 
plant with an uncertain (or misspecified)
model.
We capture the deviation of the plant dynamics 
from a linear time invariant system with an adversarial disturbance term in the system dynamics that can scale with the system state history. The amount of such deviation
that can be tolerated while maintaining system stability constitutes {\bf robustness} of the system under a given controller.

The field of adaptive control addresses the problem of
controlling linear (and non-linear) dynamical systems 
with uncertain parameters. Adaptive control algorithms
are frequently challenged on the issues of  robustness and 
transient (finite-time) performance. Here, transient performance is in contrast with asymptotic performance, and as mentioned before, robustness measures the ability to
tolerate unmodeled dynamics. A number of papers in the 1980s (e.g. 
 \citet{Rohrs1982}) pointed out a lack of robustness under model misspecification 
 for the classical {\sl model reference adaptive control} (MRAC) approach. 
One can argue that this is related to the absence of transient behavior guarantees,
such as a closed
loop $\ltgain$ bound,
with good behavior expected only asymptotically, and this is the motivation for our study.

In this paper, we show that under a fully actuated system, a properly designed 
adaptive control algorithm can exhibit a significant degree 
of robustness to unmodeled dynamics and be computationally efficient. This is in contrast to a small gain approach to analyzing robustness, where robustness is guaranteed to be inversely proportional to the $\ltgain$ of the closed loop system, excluding model misspecification. As recently shown by \cite{chen2021blackbox} via a regret lower bound, it is inevitable that the $\ltgain$  grows {\it exponentially} with the system dimension, implying a vanishing degree of robustness under the small gain theorem. 

We show that it is  possible to achieve robustness which depends {\it inverse polynomialy} on the system dimension, and independent of its parametric uncertainty, while maintaining an $\ltgain$ that grows as $2^{\O(d )}$, consistent with the known lower bounds of $2^{\tilde{\Omega}(d)}$.
Previous work by \citet{CusumanoPoollaACC1988} gives a very general, yet inefficient 
algorithm of adaptive control that achieves constant robustness for both fully actuated and under actuated systems. The algorithm assures finiteness of the
close loop $\ltgain$, but yields an excessively high $\ltgain$ bounds
(as in having $\ltgain$ that grows doubly exponentially in the dimension, in the same setting).



Our result improves upon previous work in the fully actuated setting, both in terms of computational efficiency and $\ltgain$. The controller is based on recent system identification techniques from non-stochastic control whose main component is active large-magnitude deterministic exploration. 
This technique deviates from one of the classical approaches of using least squares for system estimation and solving for the optimal controller. 
Our technique demonstrates how carefully chosen exploration for system identification can be used to bound the energy required for exploration and not to activate the system more than necessary, and yet obtain bounded $\ltgain$ up to the known lower bounds.

\subsection{Our contributions}\label{sec:contribution}
We consider the setting of a linear dynamical system with time-invariant dynamics, together with model misspecification, as illustrated in Fig.~\ref{fig:system}. 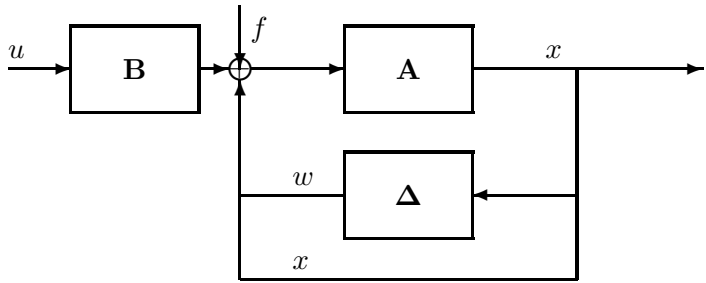
\begin{figure}[h]
\centering
\vskip5mm
\setlength{\unitlength}{0.14cm}
\begin{center}\begin{picture}(54,20)(6,42)
\thicklines
\put(0,54){\framebox(12,8){$\mathbf B$}}
\put(26,54){\framebox(12,8){$\mathbf A$}}
\put(16,58){\circle{2}}
\put(12,58){\vector(1,0){3}}
\put(-6,58){\vector(1,0){6}}
\put(17,58){\vector(1,0){9}}
\put(38,58){\vector(1,0){22}}
\put(16,64){\vector(0,-1){6}}
\put(14.9,57.3){$+$}
\put(17,61){$f$}
\put(-6,59){$u$}
\put(45,59){$x$}
\put(26,42){\framebox(12,8){$\mathbf \Delta $}}
\put(48,58){\line(0,-1){12}}
\put(48,46){\line(0,-1){8}}
\put(48,38){\line(-1,0){32}}
\put(48,46){\vector(-1,0){10}}
\put(26,46){\line(-1,0){10}}
\put(16,46){\vector(0,1){11}}
\put(16,38){\line(0,1){8}}
\put(21,47){$w$}
\put(21,39){$x$}
\end{picture}\end{center}
\caption{Diagram of the system, where $\Delta$ represents model misspecification.}
\label{fig:system}
\end{figure}

The system evolves according to the following rule, 
\begin{equation}
    x_{t+1} = A x_t +  Bu_t + \Delta_t(x_{1:t}) + f_t,
\end{equation} 
where $A, B \in \reals^{d\times d}$ is the (unknown) linear approximation to the system, $u_t , x_t, f_t \in \reals^d$ are the control, state and adversarial perturbation respectively. We refer to an upper bound on the spectral norm of $A$ as the parametric uncertainty. The perturbation $w_t = \Delta_t(x_{1:t})$ represents the deviation of the nonlinear system from the nominal system $(A, B)$. The perturbations $w_t$ crucially must satisfy the following assumption:
\begin{equation} \label{eqn:robustnessdef}
\sum_{s=1}^t \|w_s\|_2^2 \leq h^2( \sum_{s=1}^t \|x_s\|_2^2).
\end{equation}

The parameter $h$ is a measure of the robustness of the system, and is the main object of study. 
The larger $h$ is, the more model misspecification can be accommodated by the controller. Our goal is to study the limits of robustness with reasonable transient performance.  We use $\ell_2$-gain, a quantity widely studied in classical control theory, as our performance measure. The $\ell_2$-gain of a closed-loop system with control algorithm $\A$ in the feedback loop is defined as 
\begin{align}
\ltgain(\A) = \max_{f_t}   \frac{\| x_{1:T}\|_2  }{ \| f_{0:T-1} \|_2  },
\end{align}\label{eq:ltgain}
where $x_{1:T}, f_{0:T-1}\in \reals^{dT}$ are concatenations of $x_1,$ $\dots,$ $x_t,$ and $f_0, \dots, f_{T-1}$, respectively.
This notion is closely related to the competitive ratio of the control algorithm $\A$, as we show in App. \ref{app:competitive ratio}.
With this notation, we can formally state our main question: 

\medskip
\begin{frame}
\centering
\centering
{\it Is it possible to design efficient control methods that achieve robustness beyond the small gain theorem, while having $\ell_2$-gain with near-optimal \footnote{Here and elsewhere, near optimal means up to constants in the exponent.} dependence on the system dimension?}
\end{frame}

\medskip




Our study initiates an answer to this question from both lower and upper bound perspectives. In terms of upper bounds, we consider the case of a fully actuated system, and show that in this important special case, constant robustness and near-optimal $\ell_2$-gain are possible. \footnote{Obtaining similar, or even partial, results in the general under-actuated case is an exciting, important, and potentially difficult open problem, see the conclusions section.}
\begin{itemize}
        \item We give an efficient algorithm that is able to control the system with robustness $h = \Omega(\frac{1}{\sqrt{d}})$, where $d$ is the system dimension. This is independent of the parametric uncertainty.
        
       \item In addition, we show that under parametric uncertainty $M$, this algorithm achieves finite $\ltgain$ of $2^{\tilde{O}(d \log M)}$, where the dependence on system dimension is near-optimal given the lower bound of $2^{\Omega(d)}$ in \citet{chen2021blackbox}. 
        



\end{itemize}

We also consider the limits of finite $\ltgain$ and robust control. Clearly, if the system $A,B$ is not stabilizable, then one cannot obtain any lower bound on the robustness regardless of what control method is used. The distance of the system $A,B$ from being stabilizable is thus an upper bound on the robustness, and we provide a proof for completeness in App.~\ref{sec:limits on robustness}.


For our main results, we use an active explore-then-commit method for system identification and a doubling strategy to handle unknown disturbance levels. As a supplementary result, we also study system identification using the more common online least squares method, and prove that it gives constant robustness and finite $\ell_2$-gain bounds for one-dimensional systems in App.~\ref{one_dimensiona_analysis:sec}. 



\subsection{Related work}

\paragraph{Adaptive Control.} The most relevant field to our work is adaptive control, see for example the book \citep{ioannou2012robust} and survey by \citet{Tao2014}. This field has  addressed the problem of
controlling a linear dynamical system with uncertain parameters,  providing,
in the 70s, guarantees of asymptotic optimality of adaptive control algorithms.
However, reports of lack of robustness of such
algorithms to {\sl unmodeled dynamics} 
(as in the 
\citet{Rohrs1982} example)
have emerged. One can argue that this lack of robustness was due to
poor {\sl noise rejection} transient performance of such controllers, 
which can be measured in terms of {\sl $\ell_2$ induced norm (gain)} of the
overall system. The general task of designing adaptive 
controllers with finite closed loop $\ltgain$ was solved by \citet{CusumanoPoollaACC1988}, but the $\ltgain$ bounds obtained
there grow very fast with the size of parameter uncertainty,
and are therefore only good to guarantee a 
negligible amount of robustness.
It has been confirmed by \citet{megretzki_l2_gain} that even in the case of one dimensional linear models, the minimal achievable 
$\ell_2$ gain grows very fast with the size of parameter uncertainty. 
\paragraph{Nonlinear Control.}Recent research has studied provable guarantees in various complementary (but incomparable) models for nonlinear control. These include planning regret in nonlinear control \cite{agarwal2021regret}, adaptive nonlinear control under linearly-parameterized uncertainty \cite{boffi2020regret}, online model-based control with access to non-convex planning oracles \cite{kakade2020information}, control with nonlinear observation models \cite{mhammedi2020learning}, system identification for nonlinear systems \cite{mania2020active} and  nonlinear model-predictive control with feedback controllers \cite{sinha2021adaptive}. 

\paragraph{Robustness and $\ltgain$ in Control} Robust control is concerned with the ability of a controller to tolerate uncertainty in system parameters, including unmodeled dynamics present in nonlinear systems. This field has been studied for many decades, see for example \citep{robustzhou} for a survey. One fundamental method for measuring robustness is certifying stability of the closed-loop system under non-parametric uncertainty via the small gain thoerem by \cite{zames_small_gain}, where stability is implied by finite $\ltgain$. The achievability of finite $\ltgain$s for systems with unknown level of disturbance has been studied in control theory. \citet{poolla} characterize the  misspecification, or non-parametric uncertainty, tolerable for finite $\ltgain$. \citet{megretzki_l2_gain} gives a lower bound on the closed loop $\ltgain$ of adaptive controllers that achieve finite $\ltgain$ for all systems with bounded spectral norm. However, the systems studied in this paper do not contain any model misspecification.

Since the small gain theorem is known to be pessimistic, several alternative approaches have been proposed, including positivity theory and other methods of exploiting phase information of the system, constructing parameter-dependent Lyapunov functions, and using other notions of stability such as absolute stability, see \citet{beyond_small_gain} for a survey.

\paragraph{Competitive Analysis for Control} For a given controller, its $\ltgain$ is closely related to the competitive ratio, which is a quantity more often studied in the computer science community, see next section for details. \citet{yu2020competitive} gives a control algorithm with constant competitive ratio for the setting of delayed feedback and imperfect future disturbance predictions. \citet{shi2021online} proposes algorithms whose competitive ratios are dimension-free for the setting of optimization with memory, with connections to control under a known, input-disturbed system and adversarial disturbances. More recently, \citet{goel2021competitive} give an algorithm with optimal competitive ratio for known LTI systems and known quadratic costs, without misspecification.

\paragraph{System Identification for Linear Dynamical Systems.}
For an LDS with stochastic perturbations, the least squares method can be used to identify the dynamics in the partially observable and fully observable settings \citep{oymak, simchowitz18a, sarkar19a, stab}. However, least squares can lead to inconsistent solutions under adversarial disturbances, such as the model misspecification component in the system. The algorithms by \citet{simchowitz19a} and \citet{pmlr-v125-ghai20a} tolerate adversarial disturbances, but the guarantees only hold for stable or marginally stable systems. If the adversarial disturbances are bounded, \citet{hazan2019nonstochastic} and \citet{chen2021blackbox} give system identification algorithms for any unknown system, stable or not, with and without knowledge of a stabilizing controller, respectively. These techniques arose from recent results on nonstochastic control, such as works by \citet{agarwal2019online} and \citet{simchowitz2020improper}, for a comprehensive survey, see lecture notes by \citet{nsclecture}.

\subsection{Structure of the paper}

In the next section we give a few preliminaries and definitions to precisely define our setting and problem. 
In Sec.~\ref{consolidated_analysis_abs:sec} we give our main result: an efficient method with $\Omega(\frac{1}{\sqrt{d}})$ robustness and $\ltgain$ of $2^{\tilde{O}(d \log M)}$ under unknown disturbance levels.  We sketch out the analysis in Sec.~\ref{sec:analysis}.

Due to space constraints, significant technical material appears in the appendix, which can be found at \citep{FullPaper}.  App.~\ref{small_gain:sec} provides additional background on the small gain approach to robust control. In App.~\ref{sec:limits on robustness} and App.~\ref{app:competitive ratio}, we explore the limits of robustness of any controller and clarify the relationship between the performance metric $\ltgain$ and the competitive ratio, respectively. In App.~\ref{one_dimensiona_analysis:sec} we give an optimal result limited to the one-dimensional setting, where the $\ltgain$ bounds are tight in the  parametric uncertainty.   In App.~\ref{consolidated_analysis:sec} we include proofs for Sec.~\ref{consolidated_analysis_abs:sec}, and in App.~\ref{poola:sec} we provide a complete analysis of an algorithm analogous to that of \citet{CusumanoPoollaACC1988}.

\section{Preliminaries}\label{s:setup}

\paragraph{Notation.} 
We use the $\tilde{O}$ notation to hide constant and logarithmic terms in the relevant parameters. 
We use $\|\cdot\|_2$ to denote the spectral norm for matrices, and the Euclidean norm for vectors. We use $x_{s:t} \in \reals^{d(t-s+1)}$ to denote the concatenation of $x_s, x_{s+1}, \dots, x_t$, and similar notations are used for $f$, $w$, $z$. 


We make the assumptions on the model misspecification component and the disturbances in Section \ref{sec:contribution} formal.  

\begin{assumption}\label{assumption:noise_bound} We treat the model misspecification component of the system, $w_s$, as an adversarial disturbance sequence. They are arbitrary functions of past states such that for all $t$:\footnote{Notice that $w_t$ can depend on the actual trajectory of states, and not only their magnitude. This is important to capture miss-specification of the dynamics. }
$$
\|w_{1:t}\|_2   \leq h \| x_{1:t} \|_2. $$
The disturbance $f_t$ in the system is arbitrary, and let $z_t = w_t + f_t$. Without loss of generality, let $w_0 =x_0 = u_0 = 0$.
\end{assumption} 


Further, we assume the system is bounded and fully actuated.

\begin{assumption}\label{assumption:bounded_dynamics}
The magnitude of the dynamics $A, B$ are bounded by a known constant $\|A\|_2, \|B\|_2 \leq M$, where $M \ge 1$.  $B$'s minimum singular value is also lower bounded as $\sigma_{\min}(B) > L$, where $0 < L \le 1$.
\end{assumption}  

\paragraph{$\ltgain$ and Competitive Ratio.} The competitive ratio of a controller is a concept that is closely related to $\ltgain$, but is more widely studied in the machine learning community. Informally, for any sequence of cost functions, the competitive ratio is the ratio between the cost of a given controller and the cost of the optimal controller, which has access to the disturbances $f_{0:T-1}$ a priori. Importantly, the notion of competitive ratio is counterfactual: it allows for different state trajectories $x_{1:T}$ as a function of the control inputs. Under some assumptions that our algorithm satisfies, $\ltgain$ bounds can be converted to competitive ratio bounds (see Sec.~\ref{app:competitive ratio}). We choose to present our results in terms of $\ltgain$ for simplicity. 

\section{Main Algorithm and Results}\label{consolidated_analysis_abs:sec}
In this section we describe our algorithm. The main algorithm, Alg.\ref{l2_gain_multi_dim:alg}, is run in epochs, each with a proposed upper bound $q$ on the disturbance magnitude $\|f_{0:T-1}\|_2$. A new epoch starts whenever the controller implicitly discovers that $q$ is not sufficiently large and increases the upper bound.  While the disturbances $f_t$ are not directly observed, with a valid upper bound $q$, the algorithm guarantees a bounded state expansion and bounded estimates of $(A, B)$.  When these conditions are broken, we deduce that the bound on $\|f_{0:T-1}\|_2$ was incorrect and restart the system identification procedure, appropriately scaling up our upper bound $q$ .\\
The algorithm explores with large controls along the standard basis.  If the upper bound $q$ indeed exceeds $\|f_{0:T-1}\|_2$, the algorithm is guaranteed to find a stabilizing controller.  By using the standard basis vectors as the exploration set, the algorithm attains robustness depending on $\sqrt{d}$ using $O(d)$ controls. In contrast, an inefficient version of the algorithm achieves dimension-free robustness, but uses an $\epsilon$-net for exploration, resulting in an exponential number of large controls for system estimation. The alternate variant and analysis can be found in App.~\ref{consolidated_analysis:sec}.

The theorem below presents the main guarantee of our algorithm.


\begin{theorem}\label{std_basis_abs:thm}
For $h \le \frac{1}{12\sqrt{d}}$, there exists $\eps$, $\alpha$ such that Alg.~\ref{l2_gain_multi_dim:alg} has $\ltgain (\A) \le (\frac{ M d}{L})^{O(d)}$.
\end{theorem}

\begin{algorithm2e}[H]
\label{l2_gain_multi_dim:alg}
\KwIn{System upper bound $M$, control matrix singular value lower bound $L$, system identification parameter $\eps$, threshold parameter $\alpha$.} 
Set $q=0, K = 0$.\\
\While{$t \le T$}{
Observe $x_t$.\\
\eIf{ $\|x_{1:t}\|_2 > \alpha q$}{
Update $q = \|x_{1:t}\|_2$.\\
Call Alg.~\ref{sysid_multi_dim:alg} with parameters $(q, M,L, \eps, \alpha)$, obtain updated $K$ and budget $q$.}
{
Execute $u_t = -K x_t$.\\
$t \leftarrow t+1$
}}
\caption{$\ltgain$ algorithm}
\end{algorithm2e}

\begin{algorithm2e}[h]
\label{sysid_multi_dim:alg}
\KwIn{Disturbance budget $q$, system upper bound $M$, control matrix singular value lower bound $L$, system identification parameter $\eps$, threshold parameter $\alpha$.}
Call Alg.~\ref{controlid_multi_dim:alg} with parameters $(q, M,L, \eps, \alpha)$, obtain estimator $\hat{B}$ and updated budget $q$. Suppose the system evolves to time $t'= t+d$.\\
Set $q' = 4^{2d}M^{2d}\eps^{-d}q$.\\
\For{i = 0, 1, \ldots, $2d-1$}{
Observe $x_{t'+i}$.\\
\If{ $\|x_{1:t'+i}\|_2 > \alpha q$}{
Restart SysID from Line 2 with $q = \|x_{1:t'+i}\|_2$.
}
\eIf {$i$ \text{is even}}
{Play $u_{t'+i} = \xi_{i/2} \hat{B}^{-1} e_{i/2 + 1} $, $\xi_{i/2}= \frac{4^{3i/2}M^{3i/2+2}q'}{\eps^{i/2+1}}$.}
{Play $u_{t'+i} = 0$.}
}
Observe $x_{t'+2d}$, compute
$$\hat{A} = [\frac{x_{t'+2}}{\xi_0}  \cdots \frac{x_{t' + 2d}}{\xi_{d-1}}]~. $$
\If{ $\|\hat{A}\|_2 > 2M$}{
Restart SysID from Line 2 with $q = \|x_{1:t'+2d}\|_2$.
}
Return $q, K = \hat{B}^{-1}\hat{A}$\\
\caption{Adversarial System ID on Budget}
\end{algorithm2e}

\begin{algorithm2e}[H]
\label{controlid_multi_dim:alg}
\KwIn{Disturbance budget $q$, system upper bound $M$, control matrix singular value lower bound $L$, system identification parameter $\eps$, threshold parameter $\alpha$.}
\For{i = 0, 1, \ldots, $d-1$}{
Observe $x_{t+i}$.\\
\If{ $\|x_{1:t+i}\|_2 > \alpha q$}
{Restart SysID with $q = \|x_{1:t+i}\|_2$.}
Play $u_{t+i} = \lambda_i e_{i+1} $, $\lambda_i= \frac{4^{2i}M^{2i+1}q}{\eps^{i+1}}$.\\
}
Observe $x_{t+d}$, compute $$\hat{B} = [\frac{x_{t+1}}{\lambda_0} \cdots \frac{x_{t+d}}{\lambda_{d-1}}].$$ 
\If { $\|x_{1:t+d}\|_2 > \alpha q_k$ or $\sigma_{\min}(\hat{B}) < L/2$}
{Restart SysID with $q = \|x_{1:t+d}\|_2$.}
Return $q, \hat{B}$
\caption{Adversarial Control Matrix ID on Budget}
\end{algorithm2e}

\section{Analysis}\label{sec:analysis}
The algorithm has three components: exploration to estimate $B$, exploration to estimate $A$, and controlling the system with linear controller $K = \hat{B}^{-1}\hat{A}$.  The parameter $\alpha$ serves as a relative upper bound, where the state energy $\|x_{1:T}\|_2$ is guaranteed not to surpass $\alpha q$ if $q$ is a true upper bound on $\|f_{0:T-1}\|_2$.  We first analyze the case if the upper bound on the disturbance magnitude is correct and $\|f_{0:T-1}\|_2 \le q$.  In this case, the algorithm is designed with a suitable threshold $\alpha$ such that a new epoch will not be started and we are guaranteed to obtain a stabilizing controller. Note that in both exploration stages, the state can grow exponentially, so exploratory controls must also grow to keep up.  

\paragraph{Epoch Notation.} We define epochs in terms of rounds of system identification.  In particular, for the $k$th epoch, $\tstart{k}$ denotes the iteration number $t$ on the $k$th call to the system identification procedure Alg.~\ref{sysid_multi_dim:alg}, and $\tend{k} = \min(\tstart{k+1} -1, T)$ is the iteration number of the end of the epoch.  As such, within an epoch, $q$ is fixed, so we denote $q_k = \|x_{1:\tstart{k}}\|_2$ the value of $q$ within epoch $k$.

\paragraph{Identifying $B$ (see App.~\ref{sec:B_est}).} 
The first step involves identifying the control matrix using Alg.~\ref{controlid_multi_dim:alg}.  The following lemma shows that the control identification process will produce an accurate estimate of $B$ in the spectral norm with singly-exponential growth in the state energy.  Because our final controller is $K= \hat{B}^{-1} \hat{A}$, we also bound the distance of $B \hat{B}^{-1}$ from identity in order to properly stabilize the system.

\begin{lemma}\label{controlid:lem}
Suppose $\|f_{0:T-1}\|_2 \le q_k$ and $\alpha \ge 4^{2d}M^{2d}\eps^{-d}$, then running Alg.~\ref{controlid_multi_dim:alg} with $\eps \le \frac{L}{12\sqrt{d}}$ produces $\hat{B}$ such that $\|\hat{B}- B\|_2 \le 3 \eps\sqrt{d}$ and $\|B\hat{B}^{-1} - I\|_2  \le \frac{1}{2}$, with $\|x_{1:\tstart{k}+d}\|_2 \le   4^{2d}M^{2d}q_k\eps^{-d}$.
\end{lemma}

The algorithm works by probing the system with scaled standard basis vectors. With sufficiently large scaling, $x_{t+1} = Ax_t + Bu_t + z_t \approx Bu_t$.  This allows us to estimate $B$ one column at a time.  Arbitrarily large probing controls can yield an arbitrarily accurate estimate of $B$, though the magnitude of such controls will factor into the resultant $\ltgain$.  This accuracy-gain trade off is balanced deeper in the analysis.

 \paragraph{Identifying $A$ (see App.~\ref{sec:A_est}).}
 Once we have an accurate estimate of $B$, we use Alg.~\ref{sysid_multi_dim:alg} to produce an estimate $\hat{A}$ that is $O(h)$ accurate in each of the standard basis directions, again with a singly exponential state energy growth.
 
\begin{lemma}\label{approx_A:lem_2}
 Suppose $\|f_{0:T-1}\|_2 \le q_k$ and $\alpha > R= (4M)^{5d}\eps^{-2d}$, then Alg.~\ref{sysid_multi_dim:alg} produces $\hat{A}$ such that
 \begin{align*}
 \max_{i \in [d]} \|(A- \hat{A}) e_i\|_2 \le \frac{28\eps M\sqrt{d}}{L} +  3h~,
 \end{align*}
 with $\|x_{1:t'+2N}\|_2 \le  Rq_k$.
\end{lemma}
Identification of $A$ in Alg.~\ref{sysid_multi_dim:alg} works by applying controls $u_t =\xi \hat{B}^{-1}v_t$ every other iteration, where $v_t$ is a standard basis vector and $\xi$ is a large constant such that $x_{t+1} \approx Ax_t + \xi v_t + z_t \approx \xi v_t$.  One more time evolution with zero control gives $x_{t+2} = Ax_{t+1} + z_{t+1} \approx \xi Av_t + z_{t+1}$.  By Assumption~\ref{assumption:noise_bound}, $\|z_{t+1}\|_2 \le h\|x_{1:t+1}\|_2 + \|f_{0:t+1}\|_2 = O(h\xi+ q)$.  As a result, we have $\|\frac{x_{t+2}}{\xi} - Av_t\|_2 = O(h)$. By definition of $\hat{A}$ in Line 14, we also have $\|\frac{x_{t+2}}{\xi} - \hat{A}v_t\|_2 = O(h)$, so $\|(A - \hat{A})v_t\|_2 = O(h)$.  
 Exploratory controls are preconditioned with $\hat{B}^{-1}$ to achieve robustness independent of $\sigma_{\min}(B)$.
 
By exploring with the standard basis, we assure that each row of $\hat{A}$ is accurate to $O(h)$, so $\|A - \hat{A}\|_2 \le \|A - \hat{A}\|_F \le h\sqrt{d}$.  By bounding the spectral norm of the estimation error loosely through a bound on the Frobenius norm, we only produce an accurate estimate of $A$ for $h =\Omega(1/\sqrt{d})$.  With exploration complete, we shift to stabilizing the system.

\paragraph{Stabilizing the system (see App.~\ref{sec:exploitation}).}
The system is subsequently stabilized by linear controller $K = \hat{B}^{-1}\hat{A}$.  By controlling the accuracy of $\hat{A}$ and $\hat{B}$, we guarantee the closed loop system satisfies $\|A- BK\|_2 < \frac{1}{2}$ via the following simple technical lemma:

\begin{lemma}\label{K_stability:lem}
 Suppose $\|f_{0:T-1}\|_2 \le q_k$, $\alpha \ge 4^{2d}M^{2d}\eps^{-d}$, with appropriate choice of $\eps$ the resultant controller $K$ satisfies $\|A- BK\|_2 \le \frac{1}{2}$.
\end{lemma}

Now, with a stable linear system, we can bound the remaining cost of using this stabilizing controller.  In the below theorem $t^*$ represents a time such that the controller plays a stabilizing linear controller for the remainder of the time horizon.  In particular, we can view $t^*$ as the last iteration of exploration.
\begin{lemma}\label{stab_state_bound_multi_dim:lem}
If $\|f_{0: T-1}\|_2 \le q_k$, and let $t^*$ be such that $u_{t} = -K x_{t}$ for $t \ge t^* \ge \tstart{k}$, with $\|A-BK\|_2 \le 1/2$, then for $h \le \frac{1}{6}$,  
\begin{align*}
    \|x_{1:\tend{k}}\|^2_2 \le \frac{ 18 \|x_{1:t^{*}}\|^2_2 +  72 q^2_k}{7}.
\end{align*}
\end{lemma}
This follows via induction arguments involving unrolling the linear dynamics. We can then obtain the following end-to-end bound by bounding $\|x_{1:t}\|^2_2$ in terms of $q_k$, plugging in  $\|x_{1:t^*}\|_2 \le R q_k$ via the exploration analysis of Lem.~\ref{approx_A:lem_2}.

\begin{lemma}\label{std_basis_epoch_cost_main:lem}
Suppose $h \le \frac{1}{12\sqrt{d}}$, and $\eps = \frac{L}{150Md}$, then if  $\|f_{0:T-1}\|_2 \le q_k$ and $\alpha = \big(\frac{4^{14} M^{8}d^{2}}{L^{2}}\big)^d$, the running Alg.~\ref{l2_gain_multi_dim:alg} has states bounded by
\begin{align*}
    \|x_{1:\tend{k}}\|_2 \le \alpha q_k~.
\end{align*}
\end{lemma}

The restart mechanism of the algorithm eventually assures us that $q_k \approx \|f_{1:T-1}\|_2$ up to a multiplicative factor, providing an $\ltgain$ bound.

\paragraph{Handling changing disturbance budget (see App.~\ref{final_bound:sec}).}
We now sketch out the extension to unknown disturbance magnitude. In Alg~\ref{l2_gain_multi_dim:alg}, $q$ is the proposed upper bound on $\|f_{0:T-1}\|_2$. There are a variety of conditions for failure in the algorithms (i.e. where we have proof that $q$ was not a valid upper bound) which trigger re-exploration and the start of a new epoch.  If $q$ is indeed an upper bound, the above steps all will work without triggering a failure and we have $\|x_{1:T}\|_2 \leq \alpha q$ for some constant $\alpha$.  On the other hand, when a failure is detected, it is proof that $\|f_{0:T-1}\|_2 > q$.  We can relate the penultimate budget $q'$ to the final budget $q$ by bounding the state growth from a single time evolution where budget is exceeded. Combining the upper bound of $\|x_{1:T}\|_2$ and lower bound on $\|f_{0:T-1}\|_2$ produces an $\ell_2$-gain bound.



\section{Conclusions}

We have shown that for fully actuated systems, it is possible to control a misspecified LDS  with robustness that is independent of the system magnitude, going beyond the small gain theorem, with an efficient algorithm. In addition, our control algorithm has near-optimal dimension dependence in terms of $\ltgain$, improving upon the classical algorithm of \cite{poolla}.

The most important open question is to continue this investigation to the much more general case of underactuated systems. Are efficient and optimally-robust algorithms possible? 
Can an efficient algorithm can be derived to obtain constant robustness, independent of the dimension, and with a tighter bound on $\ltgain$ in terms of the system magnitude?

Other future directions include systems with partial observability and degenerate control matrices. 
It is also interesting to explore whether the same result can be obtained when the system inputs, not only the states, are subject to noise and misspecification. 

\bibliography{main}
\label{end_page}
\onecolumn


\appendix
\section{Small Gain Theorem}\label{small_gain:sec}

The Small Gain Theorem \citep{zames_small_gain} provides a guarantee on the stability on an interconnection of two stable systems, denoted $\mathbf S_{\Delta}$ and depicted in Fig.~\ref{fig:sg_system}.  System $S$ takes as input $(f, w)$ and produces output $(x,y)$ and $\mathbf \Delta$ takes as input $x$ and produces output $w$. The joint system $\mathbf S_{\Delta}$ can be viewed as taking input $f$ and producing output $y$. If the $\ltgain$ of $S$ is at most $\gamma$, the Small Gain Theorem guarantees stability of $\mathbf S_{\Delta}$ so long as the $\ltgain$ of $\mathbf \Delta$ is upper bounded by $\frac{1}{\gamma}$.  

The Small Gain Theorem is an important tool in understanding robustness. For a given closed-loop controlled system $\mathbf S$, the coupled system $\mathbf \Delta$ can viewed as model misspecification.  The Small Gain Theorem gives a prescription for robust control: design a controller such that the closed-loop system $\mathbf S$ has small $\ltgain$ and robustness follows.\\

While this methodology is appealing, unfortunately in our setting, such an approach yields quite weak bounds.  In particular, recent lower bounds for adaptive control without model misspecification scales with the parametric uncertainty as $M^{\Omega(d)}$.  As such, the best robustness we can hope for using a small-gain approach is on the order of $\frac{1}{M^d}$, vanishing as the parametric uncertainty grows.  In contrast, the algorithms in this work more directly tackle model misspecification, attaining robustness independent of the parametric uncertainty, albeit only for fully actuated systems.

\begin{figure}[h]
\centering
\vskip5mm
\setlength{\unitlength}{0.2cm}
\begin{center}\begin{picture}(54,20)(6,42)
\thicklines
\put(26,54){\framebox(12,8){$\mathbf S$}}
\put(16,60){\vector(1,0){10}}
\put(38,60){\vector(1,0){10}}
\put(38,56){\line(1,0){10}}
\put(26,42){\framebox(12,8){$\mathbf \Delta $}}
\put(48,56){\line(0,-1){10}}
\put(48,46){\vector(-1,0){10}}
\put(26,46){\line(-1,0){10}}
\put(16,46){\line(0,1){10}}
\put(16,56){\vector(1,0){10}}
\put(20,60.5){$f$}
\put(42,60.5){$y$}
\put(42,46.5){$x$}
\put(20,46.5){$w$}
\end{picture}\end{center}
\caption{Diagram of an interconnected system $\mathbf S_{\Delta}$}
\label{fig:sg_system}
\end{figure}
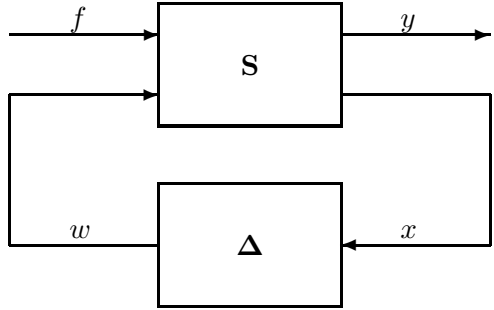
 
 For completeness, we provide a version of the Small Gain Theorem.
 \begin{theorem}[Small Gain]
 If $\ltgain$ of $\mathbf S$ is not larger than $\gamma$ and $\ltgain$ of $\mathbf \Delta$ is not larger than $\frac{1}{\gamma}$, then the $\ltgain$ of $\mathbf S_{\Delta}$ is not larger than $\gamma$.
 \end{theorem}
 
 \begin{proof}
 By the provided gain bounds, there exists constants $C_1, C_2$   such that 
 \begin{align}
     \sum_{t=1}^{\infty} \big( \gamma^2 (\|f_t\|^2_2+ \|w_t\|_2^2) - ( \|x_t\|^2 + \|y_t\|^2_2) \big) > C_1  \label{l2gain_S:eq}\\
     \sum_{t=1}^{\infty} \big( \frac{1}{\gamma^2} \|x_t\|^2_2 - \|w_t\|^2 \big) > C_2 \label{l2gain_delta:eq}
 \end{align}
 
 Scaling \eqref{l2gain_delta:eq} by $\gamma^2$ and adding to \eqref{l2gain_S:eq}, we have
 \begin{align*}
     \sum_{t=1}^{\infty} \big( \gamma^2 \|f_t\|^2_2 -  \|y_t\|^2_2 \big) > \gamma^2 C_2 +C_1~.
 \end{align*}
 Thus, the $\ltgain$ of $\mathbf S_{\Delta}$ is not larger than $\gamma$.
 \end{proof}
\section{Limits on robustness in online control}
\label{sec:limits on robustness}
In this subsection we give a simple example exhibiting the limitation of robustness, and in particular showing that in the case of an unstabilizable system, it is impossible to obtain constant robustness.

\begin{definition} [Strong Controllability]\label{def:strong-controllability}
Given a linear time-invariant dynamical system $(A, B)$, let $C_k$ denote
$$
C_k = [B\ AB\ A^2B\ \cdots A^{k-1}B] \in \mathbb{R}^{d\times kd}.
$$
Then $(A, B)$ is $(k, \kappa)$ strongly controllable if $C_k$ has full row-rank, and $\|(C_kC_k^\top)^{-1}\| \le \kappa$. 
\end{definition}

\begin{lemma}
In general, a system with strong controllability $(k,\kappa)$ cannot be  controlled with robustness larger than $\frac{1}{\sqrt{\kappa}}$. 
\end{lemma}

\begin{proof}

Consider the two dimensional system given by the matrices
$$ A_\eps = \begin{bmatrix}
2 & \eps\\
0 & 2
\end{bmatrix}  \ , \  B = \begin{bmatrix}
0 \\
1
\end{bmatrix} $$

The Kalman matrix for this system is given by 
$$ Q = [B  \ \  AB] =  \begin{bmatrix}
0 & \eps \\
1 & 2
\end{bmatrix}$$
For $\eps > 0$, this matrix is full rank, and the system is strongly controllable with parameters $(2 , O(\frac{1}{\eps^2}))$. However, for $\eps=0$, it can be seen that the system becomes uncontrollable even without any noise, since the first coordinate has no control which can cancel it, i.e. 
$ x_{t+1}(1) = 2 x_t(1) + z_t(1) . $

For adversarial noise with robustness of $\eps$, we can convert the system $A_\eps$ to $A_0$, rendering it uncontrollable. The noise sequence will simply be
$$ w_t = \begin{bmatrix}
0 & -\eps\\
0 & 0
\end{bmatrix}  x_t .$$ 
This happens with parameter $h$ which is $\eps = \frac{1}{\sqrt{\kappa}}$.
\end{proof}

\section{Relating competitive ratio to  \texorpdfstring{$\ell_2$}-gain}\label{app:competitive ratio}
As discussed, the notion of $\ltgain$ has a very similiar spirit to a competitive ratio.  Here we relate the $\ltgain$ to the competitive ratio concretely for quadratic costs. We begin with a formal definition.
\begin{definition} 
(Competitive Ratio) Consider a sequence of cost functions $c_t(x_t, u_t)$. Let $J_T(\A, f_{0:T-1})$ denote the cost of controller $\A$ given the disturbance sequence $f_{0:T-1}$, and let $\OPT( f_{0:T-1})$ denote the cost of the offline optimal controller with full knowledge of $f_{0:T-1}$. Both costs are worst case under any model misspecification that satisfies \eqref{eqn:robustnessdef} subject to a fixed $f_{0:T-1}$. 
The competitive ratio of a control algorithm $\A$, for $w_{1:T-1}$ satisfying Assumption~\ref{assumption:noise_bound} is defined as:
\begin{align*}
    \CR(\A) = \max_{f_{0:T-1}} \frac{J_T(\A, f_{0:T-1})}{\OPT(f_{0:T-1})}.
\end{align*}
\end{definition}

The $\ltgain$ bounds the ratio between $\|x_{1:T}\|_2$ and $\|f_{0:T-1}\|_2$, while under the time-invariant cost function $c_t(x, u) = \|x\|_2^2 + \|u\|_2^2$, the competitive ratio bounds the ratio of $\|x_{1:T}\|^2_2 + \|u_{1:T}\|^2_2$ to $\OPT(f_{0:T-1})$. Here we show that $\OPT(f_{0:T-1}) = \Theta( \|f_{0:T-1}\|^2_2)$, treating $M$ and $L$ as constants.  Assuming $\|u_{1:T}\|_2$ is bounded by a constant multiple of $\|x_{1:T}\|_2$, then $\mathcal{C}(\mathcal{A}) = \Theta(\ltgain(\mathcal{A})^2)$.
\begin{theorem}\label{thm:opt_lb_abstract}
Under the time-invariant cost function $c_t(x, u) = \|x\|_2^2 + \|u\|_2^2$, for any system satisfying Assumptions~\ref{assumption:noise_bound} and \ref{assumption:bounded_dynamics}, with $h < 1/2$,
\begin{align*}
    \frac{\|f_{0:T-1}\|^2_2}{9M^2}\le \OPT(f_{0:T-1})  \le \frac{8M^2\|f_{0:T-1}\|^2_2}{L^2}~.
\end{align*}
\end{theorem}

\begin{proof}
We start by bounding $\|f_t\|^2_2$ using $(a+b+c+d)^2 \le 4(a^2 + b^2 + c^2 + d^2)$.
\begin{align*}
    \|f_t\|^2_2 &= \|x_{t+1}- A x_t - B u_t-w_t\|^2_2\\ &\le 4M^2 \|x_{t}\|^2_2+ 4\|x_{t+1}\|_2^2 +  4M^2\|u_t\|^2_2 + 4\|w_t\|^2_2
\end{align*}

Summing over $f^2_t$, we have 
\begin{align*}
    \|f_{0:T-1}\|^2_2 = \sum_{t=0}^{T-1}\|f_t\|^2_2 &\le 4\sum_{t=0}^{T-1} (M^2  \|x_{t}\|^2_2 + \|x_{t+1}\|_2^2 + M^2\|u_t\|^2_2 + \|w_t\|^2_2)\\
    &\le 8M^2 (\|x_{1:T}\|^2_2 + \|u_{1:T-1}\|^2_2) + 4 \|w_{1:t}\|^2_2 \\
    &\le (8M^2 + 4h^2) (\|x_{1:T}\|^2_2 + \|u_{1:T-1}\|^2_2) ~.
\end{align*}

The lower bound follows after applying $2h< 1 \le M$.

For the upper bound, consider $u_t = -B^{-1}A x_t$, which produces closed loop dynamics $x_{t+1} = w_t + f_t$ and hence $\|x_{t+1}\|^2_2 \le 2\|w_t\|^2_2 + 2 \|f_t\|^2_2$. Summing over $t$, we have 
\begin{align*}
    \|x_{1:T}\|^2_2 \le 2\|f_{0:T-1}\|^2_2 + 2\|w_{0:T-1}\|^2_2 \le 2\|f_{0:T-1}\|^2_2 + 2h^2\|x_{0:T-1}\|^2_2~.
\end{align*}
Noting that $x_0=0$, we have $\|x_{1:T}\|^2_2 \le \frac{2\|f_{0:T-1}\|^2_2}{(1-2h^2)} \le 4\|f_{0:T-1}\|^2_2$~.

Noting that $\|u_t\|_2 \le \frac{M}{L}\|x_t\|_2$, we have 
\begin{align*}
    \|x_{1:T}\|^2_2 + \|u_{1:T-1}\|^2_2 \le \frac{2M^2\|x_{1:T}\|^2_2}{L^2} \le \frac{8M^2\|f_{0:T-1}\|^2_2}{L^2}~.
\end{align*}
\end{proof}

\begin{remark}
Dependence on $M^2$ is required in Theorem~\ref{thm:opt_lb_abstract}.  Consider the system $x_{t+1} = M x_t + u_t + f_t$ with $x_1=1$, $u_t =0$ for all $t$ and $f_t$ alternates between $-M $ and $1$.  As a result, $x_t$ oscillates between $1$ and $0$ for an average cost of $\frac{1}{2}$, while $f^2_t$ is on average  $\frac{M^2 + 1}{2}$.
\end{remark}
\section{One Dimensional Analysis}\label{one_dimensiona_analysis:sec}
In this section,we show that for a simple one dimensional system constant robustness can be achieved using certainty equivalence control (i.e online least squares system identification) with tight dependence on the system uncertainty radius. We consider the following system
\begin{equation}\label{am1}  
x_{t+1}=ax_t+u_t+w_t+f_t,\quad 
|w_{0:t}|\le h|x_{0:t}|\quad\forall~t\in\mathbb Z_+,\quad x_0=0,
\end{equation}
where $w_t$ and $f_t$ are similarly defined as in Section \ref{sec:contribution}.

In this section, we use the notation $\gamma:~[-M,M]\to[1,\infty)$ for some function
(preferably, as small as possible), such that
the following inequalities hold
\begin{equation}\label{am2}  
|x_{:t}|\le\gamma(a)|f_{:t}|\quad\forall~t\in\mathbb Z_+ 
\end{equation}
for an upper bound on the $\ltgain$.

\subsection{Lower Bound}\label{1d_lb_proof:sec}
First, we formulate (and prove) a stronger version of the result of \cite{megretzki_l2_gain},
for the case $h=0$ (which means that $w\equiv0$ in (\ref{am1})).

\begin{theorem}\label{1d_lb:thm}
 If a control algorithm $\A$ has finite $\ell_2$-gain bounds specified by $\gamma: [-M, M] \to [0, \infty)$ with $h=0$, for all $a \in [-M, M]$ and $f_t,w_t,x_t\in\R$ satisfying \eqref{am1}, then 
$ \gamma(a) \ge \max\{|a-M|,|a+M|\}/8\ge M/8$ for all $a\in [-M, M]$, and therefore $\ltgain(\A)\ge M/8$.
\end{theorem}

The result of Thm.~\ref{1d_lb:thm} suggests that the closed loop
$\ltgain$, no matter which adaptive controller is used, if it obtains finite $\ltgain$ for all systems satisfying the parametric uncertainty, must grow 
linearly with the parametric uncertainty size for all values of 
the uncertain parameter.
In particular, this makes it impossible to sacrifice performance at 
some values of $a$ to gain much improvement at other values of $a$.
Given signals $x,f$ satisfying equations \eqref{am1} with $w_t\equiv0$, define
$\xi_t,p_t,q_t,r_t$ by
\[ \xi_t=x_{t+1}-u_t,\qquad p_t=\sum_{s=0}^{t-1}x_s^2,\qquad 
q_t=\sum_{s=0}^{t-1}x_s(x_{s+1}-u_s),\qquad r_t=\sum_{s=0}^{t-1}\xi_s^2.\]
The interaction between control ($u_t$) and noise ($f_t$) can
be viewed as a game (between $u$ and $f$), 
in which $u_t$ is decided, based on knowing all $x_s$
with $s\le t$ (and, of course, all $u_s$ with $s<t$), to keep the inequality 
\begin{equation}\label{am11}   
\gamma(a)^2\left[p_ta^2-2q_ta+r_t\right]\ge p_t
\end{equation}
satisfied for all $a\in[-M,M]$ and all $t\in\mathbb Z_+$, while $f_t$ is decided based on
knowing all $u_s$ and $x_s$ with $s\le t$ (and all $f_s$ with $s<t$), 
in an effort to violate the inequality 
from (\ref{am11}) at some time $t\in\mathbb Z_+$ and some $a\in[-M,M]$. 

\vskip4mm\noindent
In this proof, we work with the {\sl normalized} versions 
$\delta_t$, $\nu_t$, $z_t$ of $u_t$, $\xi_t$, $x_t$, as well as additional signals
$\beta_t$, $\theta_t$. Let $t_0$ be the smallest $t\in\mathbb Z_+$ such that $p_t>0$.
For $t\ge t_0$ let
\[ \beta_t=\frac{q_t}{p_t},\qquad\theta_t=\frac{r_tp_t-q_t^2}{p_t^2},\qquad
z_t=\frac{x_t}{\sqrt{p_t}},\qquad\delta_t=\frac{u_t+\beta_tx_t}{\sqrt{p_t+x_t^2}},\qquad
\nu_t=\frac{\xi_t-\beta_tx_t}{\sqrt{p_t+x_t^2}}.\]
Since $x_t,\xi_t,u_t,p_t,q_t,r_t$ satisfy equations
\[  x_{t+1}=\xi_t+u_t,~p_{t+1}=p_t+x_t^2,~q_{t+1}=q_t+x_t\xi_t,~r_{t+1}=r_t+\xi_t^2,\]
$\delta_t$, $\nu_t$, $z_t$, $\beta_t$, $\theta_t$, for $t\ge t_0$, satisfy
\begin{equation}\label{am12}
z_{t+1}=\nu_t+\delta_t,\qquad
\beta_{t+1}=\beta_t+\frac{z_t\nu_t}{\sqrt{1+z_t^2}},\qquad
\theta_{t+1}=\frac{\theta_t+\nu_t^2}{1+z_t^2}.
\end{equation}
The game between $u$ and $f$ can now be interpreted as the game between
$\delta$ and $\nu$, defined by the dynamical equations (\ref{am12}) with the state
$y_t=(z_t,\beta_t,\theta_t)$,
where the normalized control effort
$\delta_t$, for $t\ge t_0$, is best constructed as a function $\delta_t=S_t(y_t)$
of the current state, to keep the inequality
\begin{equation}\label{am13}   
\gamma(a)^2\left[(a-\beta_t)^2+\theta_t\right]\ge1
\end{equation}
satisfied for all $a\in[-M,M]$ and all $t\ge t_0$, while $\nu_t$,  for $t\ge t_0$, 
is best constructed as a function $\nu_t=D_t(y_t,\delta_t)$ 
of current normalized control effort and the current state,
in an effort to violate the inequality 
from (\ref{am13}) at some time $t\ge t_0$ and some $a\in[-M,M]$. 

\paragraph{Claim 1:} {\sl If a control algorithm $\delta_t=S_t(y_t)$ maintains (\ref{am13})
for $a=a_0$ and for all $t\ge t_0$ then it also satisfies
\begin{equation}\label{am17}   
\gamma(a_0)^2\left[(a_0-\beta_t)^2+\theta_t\right]\ge1+z_t^2
\end{equation}
for all $t\ge t_0$.} Indeed, with $\nu_t=\frac{(a_0-\beta_t)z_t}{\sqrt{1+z_t^2}}$ we have
\[
\gamma(a_0)^2\left[(a_0-\beta_{t+1})^2+\theta_{t+1}\right]=
\gamma(a_0)^2\left[\left|a_0-\beta_{t}-\frac{z_t\nu_t}{\sqrt{1+z_t^2}}\right|^2+
\frac{\theta_t+\nu_t^2}{1+z_t^2}\right]
=\gamma(a_0)^2\frac{(a_0-\beta_{t})^2+\theta_{t}}{1+z_t^2},
\]
hence (\ref{am17}) must be satisfied to maintain (\ref{am13}).

\paragraph{Claim 2:} {\sl If a control algorithm $\delta_t=S_t(y_t)$ 
maintains (\ref{am17})
for all $t\ge t_0$ then it also satisfies
\begin{equation}\label{am18}   
\gamma(a_0)^2\left[\frac{\theta_t}{1+z_t^2}+\frac{(a_0-\beta_t)^2}{z_t^2}\right]-1\ge
\left|\delta_t+\frac{(a_0-\beta_t)\sqrt{1+z_t^2}}{z_t}\right|^2
\end{equation}
whenever $z_t\neq0$, for all $t\ge t_0$.}  Indeed, when $z_t\neq0$, with 
$\nu_t=\frac{(a_0-\beta_t)\sqrt{1+z_t^2}}{z_t}$ we have
\[
\gamma(a_0)^2\left[(a_0-\beta_{t+1})^2+\theta_{t+1}\right]-1-z_{t+1}^2
=\gamma(a_0)^2\left[\frac{\theta_t}{1+z_t^2}+\frac{(a_0-\beta_t)^2}{z_t^2}\right]-1-
\left|\delta_t+\frac{(a_0-\beta_t)\sqrt{1+z_t^2}}{z_t}\right|^2,
\]
hence (\ref{am18}) must be satisfied to maintain (\ref{am17}).

\vskip4mm\noindent
To continue the proof,
take any point $a_0\in[-M,M]$, 
and let $\gamma_0=\gamma(a_0)\ge1$.
We aim to show that, for a sufficiently small $\mu>0$,
\begin{itemize}
\item[(A)] an appropriate adversary strategy $\nu_t=D_t(y_t,\delta_t)$, 
assures that $\theta_t\to0$,
while $\beta_t$ stays (for sufficiently large $t$) 
within the interval $[a_0+(4+\mu)\gamma_0,a_0+(8+3\mu)\gamma_0]$,
no matter which normalized control algorithm $\delta_t=S_t(y_t)$ is used. 
\item[(B)]  an appropriate adversary strategy $\nu_t=D_t(y_t,\delta_t)$, 
assures that $\theta_t\to0$,
while $\beta_t$ stays (for sufficiently large $t$) 
within the interval $[a_0-(8+3\mu)\gamma_0,a_0-(4+\mu)\gamma_0]$,
no matter which normalized control algorithm $\delta_t=S_t(y_t)$ is used.
\end{itemize}
Combining assertions (A) and (B) (and making $\mu>0$ sufficiently small) 
assures that, as long as $\gamma(a)$ is finite for
all $a\in[-M,M]$, the value of $8\gamma(a)$ cannot be 
smaller than the distance from $a$ to
either end of the $[-M,M]$ interval, implying the 
desired lower bound for $\gamma(a)$. 

\vskip4mm\noindent
We are only presenting
the strategy for selecting $\nu_t$ in (A), as the construction for (B) is symmetric.
Let $I_0$ denote the interval $[a_0+(4+\mu)\gamma_0,a_0+(8+3\mu)\gamma_0]$.
For $t\ge t_0$, define $\nu_t$ by
\begin{eqnarray}
\nu_t =& 0,&\text{when}~|\delta_t|\ge2\gamma_0,~\beta_t\in I_0,\label{am14}\\
\nu_t =& -\mu\gamma_0\text{sign}[z_t(\beta_t-a_0-(6+2\mu)\gamma_0)],
& \text{when}~|\delta_t|\ge(2+\mu)\gamma_0,~\beta_t\not\in I_0,\label{am15}\\
\nu_t =& -(4+\mu)\gamma_0\text{sign}[z_t(\beta_t-a_0-(6+2\mu)\gamma_0)],& \text{otherwise,}\label{am16}
\end{eqnarray}
where the \enquote{sign} function is defined by 
$\text{sign}(x)=1$ for $x\ge0$, $\text{sign}(x)=-1$ for $x<0$, 
i.e., only takes values $1$ or $-1$. Intuitively, the adversary strategy 
(\ref{am14})-(\ref{am16}) pursues the following three objectives:
\begin{itemize}
\item[(a)] keep $|z_t|\ge2\gamma_0$ at all times;
\item[(b)] force $\beta_t\in I_0$, eventually;
\item[(c)] make $\nu_t=0$ when objectives (a) and (b) are satisfied.
\end{itemize}

\paragraph{Claim 3:} {\sl Subject to (\ref{am12}) and (\ref{am14})-(\ref{am16}), 
condition
$|z_{t}|\ge2\gamma_0$ will be satisfied  for all $t>t_0$.}   Indeed, 
\begin{description}
\item[\quad(3a):] in (\ref{am14}),
$|\delta_t|\ge2\gamma_0$ and $\nu_t=0$, hence 
$|z_{t+1}|=|\delta_t+\nu_t|=|\delta_t|\ge2\gamma_0$.
\item[\quad(3b):] in (\ref{am15}), $|\delta_t|\ge(2+\mu)\gamma_0$ and 
$|\nu_t|=\mu\gamma_0$, hence 
$|z_{t+1}|=|\delta_t+\nu_t|\ge|\delta_t|-|\nu_t|\ge2\gamma_0$.
\item[\quad(3c):] in (\ref{am16}), $|\delta_t|\le(2+\mu)\gamma_0$ and
$|\nu_t|=(4+\mu)\gamma_0$, hence
$|z_{t+1}|=|\delta_t+\nu_t|\ge|\delta_t|-|\nu_t|\ge2\gamma_0$.
\end{description}

\paragraph{Claim 4:} {\sl Subject to (\ref{am12}), (\ref{am14})-(\ref{am16}), and
assuming $\mu\in(0,4)$,
there exists $t_1>t_0$ such that $\theta_t<4+3\mu$ for all $t>t_1$.}
Indeed, since $|\nu_t|\le(4+\mu)\gamma_0$ and
$|z_{t}|\ge2\gamma_0$ for all $t>t_0$, we have
\[  \theta_{t+1}-\frac{(4+\mu)^2}4
=\frac{\theta_t+\nu_t^2}{1+z_t^2}-\frac{(4+\mu)^2}4
\le\frac{\theta_t+(4+\mu)^2\gamma_0^2}{1+4\gamma_0^2}-\frac{(4+\mu)^2}4
=\frac1{1+4\gamma_0^2}\left(\theta_{t}-\frac{(4+\mu)^2}4\right)\]
for all $t>t_0$,
which leads to the conclusion, since $4+3\mu>(4+\mu)^2/4$ for all $\mu\in(0,4)$.

\paragraph{Claim 5:} {\sl If $0<\mu<\gamma_0^{-2}/3$, and
$\beta_t\in I_0$ for some $t>t_1$ then 
$|\delta_t|\ge2\gamma_0$ (and therefore $\nu_t=0$, $\beta_{t+1}=\beta_t$).}
Indeed, combining $|z_t|\ge2\gamma_0$ with $\theta_t<4+3\mu$ shows that
\[  \frac{\gamma_0^2\theta_t}{1+z_t^2}-1
\le\frac{\gamma_0^2(4+3\mu)}{1+4\gamma_0^2}-1
=\frac{3\mu\gamma_0^2-1}{1+4\gamma_0^2}<0,\]
hence, using (\ref{am18}),
\begin{eqnarray*}|\delta_t|&\ge&\left|\frac{(\beta_t-a_0)\sqrt{1+z_t^2}}{z_t}\right|
-\left|\delta_t-\frac{(\beta_t-a_0)\sqrt{1+z_t^2}}{z_t}\right|\\
&\ge&\frac{|\beta_t-a_0|\sqrt{1+z_t^2}}{|z_t|}-\frac{|\beta_t-a_0|\gamma_0}{|z_t|}
=|\beta_t-a_0|\frac{\sqrt{1+z_t^2}-\gamma_0}{|z_t|}.
\end{eqnarray*}
Since, for $\gamma_0\ge1$, the function $\phi:~(0,\infty)\to\mathbb R$ defined by
$\phi(z)=\frac{\sqrt{1+z^2}-\gamma_0}{z}$ has positive derivative
\[  \dot\phi(z)=\frac1{z^2}\left(\gamma_0-\frac1{\sqrt{1+z^2}}\right)>0,\]
it is monotonically increasing
on $(0,\infty)$, which transofrms the lower bound for $|\delta_t|$ into
\[  |\delta_t|\ge |\beta_t-a_0|\frac{\sqrt{1+z_t^2}-\gamma_0}{|z_t|}\ge
(4+\mu)\gamma_0\frac{\sqrt{1+4\gamma_0^2}-\gamma_0}{2\gamma_0}>
(2+\mu/2)\gamma_0>2\gamma_0.\]

\vskip4mm\noindent
Claim  5 establishes that, once $\beta_\tau\in I_0$ for some $\tau>t_1$, the equalities
$\beta_t=\beta_\tau\in I_0$ and $\nu_t=0$ will hold for all $t\ge\tau$, which will
guarantee that 
$\theta_{t+1}=\theta_t/(1+z_t^2)\le\theta_t/(1+4\gamma_0^2)\le\theta_t/5$ will converge
to zero as $t\to+\infty$, thus preventing (\ref{am13}) with $a=\beta_\tau$
from being satisfied
when $t$ is large enough, no matter how large $\gamma(\beta_\tau)$ is.

\vskip4mm\noindent
To finish the proof, we need to show that condition $\beta_t\in I_0$ will be satisfied for some
$t>t_1$. Indeed, with $\beta_t\not\in I_0$, the value of $\nu_t$ will be defined by
either (\ref{am15}) or (\ref{am16}). Taking into accout that, for $|z_t|\ge2\gamma_0\ge2$, 
the value of $|z_t|/\sqrt{1+z_t^2}$ is between 0.5 and 1,
\begin{description}
\item[\quad Case 1:] according to (\ref{am15}) and (\ref{am12}), $\beta_{t+1}$ 
results from moving $\beta_t$, which is at least $(2+\mu)\gamma_0$ away from the center
$c_0=a_0+(6+2\mu)\gamma_0$ of $I_0$, by a distance between $0.5\mu\gamma_0$ and 
$\mu\gamma_0$ in the direction \enquote{to the center}, which ensures that
$|\beta_{t+1}-c_0|\le|\beta_t-c_0|-\mu/2$.
\item[\quad Case 2:] according to (\ref{am16}) and (\ref{am12}), $\beta_{t+1}$ 
results from moving $\beta_t$, which is at least $(2+\mu)\gamma_0$ away from
$c_0$, by a distance between $(2+\mu/2)\gamma_0$ and 
$(4+\mu)\gamma_0$ in the direction \enquote{to the center}, which ensures that
$|\beta_{t+1}-c_0|\le|\beta_t-c_0|-\mu$.
\end{description}
Therefore, as long as $t>t_1$ and $\beta_{t}\not\in I_0$, the value of $|\beta_t-c_0|$
decreases by at least $\mu/2$ at each step, thus guaranteeing that condition
$\beta_t\in I_0$ will be satisfied for some $t>t_1$.

\subsection{Upper Bound}\label{cert_equiv_proof:sec}

In this section, we show that, for $h<1/2$, 
the so-called {\sl certainty equivalence controller}
achieves an upper bound for the closed loop $\ell_2$ gain
which grows linearly with $M$. This
is remarkable, as a large $\ltgain$ is usually associated with low robustness
to uncertain dynamical feedback, with an expectation that a system with $\ltgain$ 
$\gamma$ can be destabilized by a feedback with 
$\ltgain$ of $1/\gamma$ (which is
certainly true for {\sl linear time invariant} systems). Since, with $h=0$, $\ltgain$
of the closed loop is shown to be at least $M/8$, one would expect that 
some uncertainty
$\mathbf\Delta$ of $\ltgain$ $h=8M^{-1}$ would destabilize the
system, but this is evidently not happening: 
according to Thm.~\ref{cert_equiv:thm} below, all uncertainty $\ltgain$s below $h=0.5$ are
well tolerated by the certainty equivalence controller (Alg.~\ref{alg:cert_equiv}).

\begin{algorithm2e}[h]

\label{alg:cert_equiv}
\KwIn{Time horizon $T$, system upper bound parameter $M$.}
Initialize $x_0 , u_0 =0$\\
\For{$t=1\dots T$}{
Observe $x_t$ and define $\tilde{z}_{t-1}(\hat{a}) = x_t -\hat{a} x_{t-1} - u_{t-1}$.\\
Compute $\bar{a}_t = \argmin_{\hat{a}} \sum_{s=0}^{t-1} \tilde{z}^2_s(\hat{a})$\\
Compute $\hat{a}_t = \text{clip}_{[-M, M]}(\bar{a}_t)$\\
Execute $u_t = -\hat{a}_t x_t$.
}
\caption{Certainty Equivalence Control}
\end{algorithm2e}

\vskip4mm\noindent

\begin{theorem}
\label{cert_equiv:thm}
 For $M\ge1/4$, with a system satisfying condition \eqref{am1} with $h<1/2$, Alg.~\ref{alg:cert_equiv} has
 \begin{align*}
     \ltgain(\A) \le \frac{(64M^2-8h)^{1/2}}{(1-2h)^{3/2}}~.
 \end{align*}
\end{theorem}

We provide an intuitive explanation why certainty equivalence can provide an $\ltgain$ bound.  If the algorithm estimates $\hat{a}$ inaccurately, strong convexity of the one dimensional least squares objective implies that the magnitude of the disturbances is a nontrivial fraction of the magnitude of the states up to that point.  On the other hand, if $\hat{a}_{t+1}$ is an accurate estimate of $a$, we can bound $\|x_{1:t}\|^2_2$ using the stability of the closed loop dynamics.  An $\ltgain$ bound follows from stitching these regimes together.  While we would like to extend these ideas to high dimensions, we note that the least squares objective is no longer strongly convex in such a setting.  In particular, $\|A-\hat{A}_t\|_2$ can be large in a direction where disturbances are small relative to the magnitude of the state.   A more technical approach that yields the tighter bound  of Thm.~\ref{cert_equiv:thm} can be found in App.~\ref{cert_equiv_proof:sec}.
We restate Alg.~\ref{alg:cert_equiv} in closed form:
\begin{equation}\label{am3}
\tilde a_t=\text{clip}_{[-M, M]}(\hat a_t),\quad
\hat a_t=\begin{cases}Q_t/Y_t,& Y_t\neq0,\\ 0,& X_t=0,\end{cases}\qquad
X_t=\sum_{s=0}^{t-1}x_s^2,\quad Q_t=\sum_{s=0}^{t-1}(x_{s+1}-u_s)x_s,
\end{equation}

Note that $\hat a_t$ is the argument of minimum (with respect to $a\in\mathbb R$)
of the \enquote{equation error} functional
\begin{equation}\label{am5} 
V_t(a)=\sum_{s=0}^{t-1}|x_{t+1}-ax_t-u_t|^2=\sum_{s=0}^{t-1}|w_s+f_s|^2,
\end{equation}
computed with the data available to the controller at time $t$.

Let $f,w,x\in\ell$ be some signals
satisfying (\ref{am1}), where, $M\ge1/4$,  and $h<1/2$.
Let $v=f+w$.
In addition to variables $x_t$, $Q_t$ defined in (\ref{am3}), define $V_t=V_t(a)$
as in (\ref{am5}), and let
\[  H_t=x_t^2,\qquad Z_t=Z_t(a)=\max_{\tau\le t}\left\{V_\tau(a)-hx_\tau/2\right\},\qquad
R_t=\sum_{s=0}^{t-1}f_s^2,\qquad W_t=\sum_{s=0}^{t-1}w_s^2.\]
In contrast with $x_t$, $Q_t$, and $H_t$, the values $V_t=V_t(a)$ and $Z_t=Z_t(a)$
are not available to the controller, as they depend not only on the past observations
$x_{0:t}$, but on the unknown parameter $a\in[-M,M]$, and  
variables $R_t$ and $W_t$
depend not only on the past observations
$x_{0:t}$ and  $a\in[-M,M]$, but also on the unknown dynamics of 
$\mathbf\Delta$. 

\paragraph{Step 1:} Show that the relation between $H_t$,
$x_t$, $V_t$, $Z_t$, $v_t$, and
$H_{t+1}$, $x_{t+1}$, $V_{t+1}$, $Z_{t+1}$ can be described, for all $t\in\mathbb Z_+$, by
\begin{eqnarray}\label{am6}
H_{t+1}&\le & \min\left\{8M^2,2V_t/x_t\right\}H_t+2v_t^2,\quad H_0=0,\label{am6a}\\
x_{t+1}&=&x_t+H_t,\quad x_0=0,\label{am6b}\\
V_{t+1}&=&V_t+v_t^2,\quad V_0=0,\label{am6c}\\
Z_{t+1}&=&\max\left\{Z_t,~V_{t+1}-hx_{t+1}/2\right\},\quad Z_0=0.\label{am6d}
\end{eqnarray}
Indeed, equations (\ref{am6b})-(\ref{am6d}) are evident. 
To prove (\ref{am6a}), note first that,
by the definition of $\hat a_t$,
\[  V_t(a)=(a-\hat a_t)^2x_t+\min_{b\in\mathbb R}V_t(b)\qquad\text{for all}
\quad a\in\mathbb R,\]
hence $V_t\ge(a-\hat a_t)^2x_t$. Since $|a-\text{clip}_{-M,M]}(b)|\le|a-b|$ for every 
$a\in[-M,M]$ and $b\in\mathbb R$, applying this to $b=\hat a$
yields $V_t\ge(a-\tilde a_t)^2x_t$. Also, since both $a$ and 
$\tilde a_t$ are in $[-M,M]$,
we have $|a-\tilde a_t|\le2M$. Since the quadratic form 
$\sigma(p,q)=2p^2+2q^2-(p+q)^2$ is positive semidefinite, we have
\[  x_{t+1}^2=[(a-\tilde a_t)x_t+v_t]^2\le2(a-\tilde a_t)^2x_t^2+2v_t^2
\le2\min\left\{4M^2, V_t/x_t\right\}x_t^2+2v_t^2.\]

\paragraph{Step 2:} Show that 
\begin{equation}\label{am4}
Z_t\le\frac1{1-2h}R_t\qquad\text{for all}\qquad t\in\mathbb Z_+.
\end{equation}
Indeed, for $h=0$ we have $w_t\equiv0$, hence $V_t=R_t$, and the
inequality (\ref{am4}) holds. For $h\in(0,1/2)$, the quadratic form
\[   \sigma(f,w)=\frac1{2h}w^2+\frac1{1-2h}f^2-(f+w)^2\]
is positive semidefinite. Hence, realizing that the L2 gain bound 
$|w_{0:t}|\le h|x_{0:t}|$ means $h^{-2}W_{t+1}\le x_{t+1}$,
we get
\[  V_t-0.5hx_t\le V_t-\frac1{2h}W_t
=\sum_{s=0}^{t-1}\left\{(f_s+w_s)^2-\frac1{2h}w_s^2\right\}\le
\frac1{1-2h}\sum_{s=0}^{t-1}f_s^2=\frac1{1-2h}R_t.\]
Finally, since $R_t$ is monotonically non-decreasing as $t$ increases,
\[ Z_t=\max_{\tau\le t}\left\{V_\tau-hx_\tau/2\right\}\le\max_{\tau\le t}\frac1{1-2h}R_\tau
=\frac1{1-2h}R_t.\]

\paragraph{Step 3:} Use (\ref{am6a})-(\ref{am6c}) to show that
\begin{equation}\label{am7}
D_t\stackrel{\text{def}}{=}H_t-2V_t-(8M^2-1)x_t\le0\qquad
\text{for all}\quad t\in\mathbb Z_+.
\end{equation}
Indeed, $D_0=H_0-2V_0-(8M^2-1)x_0=0\le0$, and bounding $D_{t+1}$ in terms of $D_t$
yields
\[ H_{t+1}-2V_{t+1}-(8M^2-1)x_{t+1}\le
8M^2H_t+2v_t^2-2(V_t+v_t^2)-(8M^2-1)(x_t+H_t)=H_t-2V_t-(8M^2-1)x_t.\]

\paragraph{Step 4:} Use (\ref{am6a})-(\ref{am6d}) to show that
\begin{equation}\label{am8}
C_t\stackrel{\text{def}}{=}\rho Z_t-2H_t-x_t+4V_t\ge0\quad
\text{for all}\quad t\in\mathbb Z_+,
\qquad\text{where}
\qquad\rho=\frac{64M^2-4}{1-2h}.
\end{equation}
Indeed, the inequality is evidently satisfied for $t=0$. Assuming that 
$\rho Z_t-2H_t-x_t+4V_t\ge0$:

\paragraph{Case 4a:} If $x_t\le M^{-2}V_t/4$ then 
$Z_t\ge V_t-hx_t/2\ge (1-hM^{-2}/8)V_t\ge (64M^2-4)\rho^{-1}V_t$, hence
\begin{eqnarray*}
C_{t+1}&\ge&
\rho Z_t-2(8M^2H_t+2v_t^2)-(x_t+H_t)+4(V_t+v_t^2)\\
&=&\rho Z_t-(16M^2+1)H_t-x_t+4V_t\\
&\ge&\rho Z_t-(16M^2+1)[2V_t+V_t(8M^2-1)M^{-2}/4]-V_tM^{-2}/4+4V_t\\
&=&\rho Z_t-(64M^2-4)V_t\ge0.
\end{eqnarray*}

\paragraph{Case 4b:} If $M^{-2}V_t/4\le x_t\le 4$ then 
$Z_t\ge V_t-hx_t/2\ge(1-2h)V_t=(64M^2-4)\rho^{-1}V_t$, hence
\begin{eqnarray*}
C_{t+1}&\ge&\rho Z_t-2[2H_tV_t/x_t+2v_t^2]-(x_t+H_t)+4(V_t+v_t^2)\\
&\ge& \rho Z_t-[2V_t+(8M^2-1)x_t](1+4V_t/x_t)-x_t+4V_t\\
&=&\rho Z_t-(32M^2-6)V_t-8M^2x_t-8V_t^2/x_t.
\end{eqnarray*}
The last expression is a concave function of $x_t$, hence its values, with 
$x_t$ ranging over the interval $[M^{-2}V_t/4,~4]$, are not smaller than its values at the 
ends of the interval, which are both equal to $\rho Z_t-(64M^2-4)V_t\ge0$

\paragraph{Case 4c:} if $x_t\ge 4$ then 
\[  C_{t+1}\ge\rho Z_t-2[2H_tV_t/(4V_t)+2v_t^2]-(x_t+H_t)+4(V_t+v_t^2)=C_t\ge0.\]

\paragraph{Step 5:} using the inequalities $\rho Z_t\ge2H_t+x_t-4V_t$ and
$Z_t\ge V_t-hx_t/2$, for $\delta=(1-2h)\rho/(4+\rho)$ we have either

\paragraph{Case 5a:} $(1-\delta)x_t\ge 4V_t$, in which case $Z_t\ge\delta\rho^{-1}x_t
=\frac{1-2h}{4+\rho}x_t$, or

\paragraph{Case 5b:} $(1-\delta)x_t\le 4V_t$, in which case 
$Z_t\ge V_t-hx_t/2\ge\frac{1-\delta-2h}{4}x_t=\frac{1-2h}{4+\rho}x_t$.

\vskip4mm\noindent
Combining these two observations together, we have
\[  x_t\le\frac{4+\rho}{1-2h}Z_t\le \frac{4+\rho}{(1-2h)^2}R_t=\frac{64M^2-8h}{(1-2h)^3}R_t,\]
which proves that the closed loop L2 gain from $f$ to $y$ is not larger than
$\frac{(64M^2-8h)^{1/2}}{(1-2h)^{3/2}}$.
\section{Full Analysis}\label{consolidated_analysis:sec}
In this section we provide a complete analysis of our main algorithm.  We present the algorithm including a generic exploration set and provide an alternate analysis when all control directions in an $\eps$-net are explored.
We denote an $\eps$-net as $\mathcal{N}_{\eps,d}$, defined as:
\begin{definition}\label{def:eps_net}
We define $\mathcal{N}_{\eps,d} \subseteq \R^d$ to be an $\eps$-net of  $\mathbb{S}^{d-1}$, the unit sphere with the euclidean metric, if for any $x \in \mathbb{S}^{d-1}$, we have $x' \in \mathcal{N}_{\eps,d}$ such that $\|x-x'\|_2 \le \eps$.
\end{definition}

\begin{algorithm2e}[h]
\label{l2_gain_multi_dim_gen:alg}
\KwIn{System upper bound $M$, control matrix singular value lower bound $L$, system identification parameter $\eps$, threshold parameter $\alpha$, and exploration set $V \subseteq \mathbb{S}^{d-1}$.} 
Set $q=0, K = 0$.\\
\While{$t \le T$}{
Observe $x_t$.\\
\eIf{ $\|x_{1:t}\|_2 > \alpha q$}{
Update $q = \|x_{1:t}\|_2$.\\
Call Alg.~\ref{sysid_multi_dim_gen:alg} with parameters $(q, M,L, \eps, \alpha, V)$, obtain updated $K$ and budget $q$.}
{
Execute $u_t = -K x_t$.\\
$t \leftarrow t+1$
}}
\caption{$\ltgain$ algorithm}
\end{algorithm2e}

\begin{algorithm2e}[h]
\label{sysid_multi_dim_gen:alg}
\KwIn{Disturbance budget $q$, system upper bound $M$, control matrix singular value lower bound $L$, system identification parameter $\eps$, threshold parameter $\alpha$, and exploration set $V \subseteq \mathbb{S}^{d-1}$.}
Define $N = |V| \geq d$ with $V = (v_0, v_1, \dots, v_{N-1})$.\\
Call Alg.~\ref{controlid_multi_dim_gen:alg} with parameters $(q, M,L, \eps, \alpha)$, obtain estimator $\hat{B}$ and updated budget $q$. Suppose the system evolves to time $t'= t+d$.\\
Set $q' = 4^{2d}M^{2d}\eps^{-d}q$.\\
\For{i = 0, 1, \ldots, $2N-1$}{
Observe $x_{t'+i}$.\\
\If{ $\|x_{1:t'+i}\|_2 > \alpha q$}{
Restart SysID from Line 2 with $q = \|x_{1:t'+i}\|_2$.
}
\eIf {$i$ is even}
{Play $u_{t'+i} = \xi_{i/2} \hat{B}^{-1} v_{i/2} $, $\xi_{i/2}= \frac{4^{3i/2}M^{3i/2+2}q'}{\eps^{i/2+1}}$.}
{Play $u_{t'+i} = 0$.}
}
Observe $x_{t'+2N}$, compute
$$\hat{A} \in \argmin_{\tilde{A}: \|\tilde{A}\|_2 < M} \Phi(\tilde{A}):=\max_{i \in [0,  N)} \|\tilde{A}v_{i} - \frac{x_{t'+2i+2}}{\xi_{i}}\|_2~. $$
Return $q, K = \hat{B}^{-1}\hat{A}$\\
\caption{Adversarial System ID on Budget}
\end{algorithm2e}

\begin{algorithm2e}[h]
\label{controlid_multi_dim_gen:alg}
\KwIn{Disturbance budget $q$, system upper bound $M$, control matrix singular value lower bound $L$, system identification parameter $\eps$, threshold parameter $\alpha$.}
\For{i = 0, 1, \ldots, $d-1$}{
Observe $x_{t+i}$.\\
\If{ $\|x_{1:t+i}\|_2 > \alpha q$}
{Restart SysID with $q = \|x_{1:t+i}\|_2$.}
Play $u_{t+i} = \lambda_i e_{i+1} $, $\lambda_i= \frac{4^{2i}M^{2i+1}q}{\eps^{i+1}}$.\\
}
Observe $x_{t+d}$, compute $$\hat{B} = [\frac{x_{t+1}}{\lambda_0} \cdots \frac{x_{t+d}}{\lambda_{d-1}}].$$ 
\If { $\|x_{1:t+d}\|_2 > \alpha q_k$ or $\sigma_{\min}(\hat{B}) < L/2$}
{Restart SysID with $q = \|x_{1:t+d}\|_2$.}
Return $q, \hat{B}$
\caption{Adversarial Control Matrix ID on Budget}
\end{algorithm2e}

\begin{remark}
We note that when $V$ is the standard basis, $\hat{A}$ has the closed form used in Alg.\ref{sysid_multi_dim:alg}.  In particular, the unconstrained solution\footnote{With small modifications to analysis, the constrained optimization can be replaced by a failure check if $\|\hat{A}\|_2 > 2M$ as this would indicate our disturbance budget is too small.} of Line 16 in Alg.~\ref{sysid_multi_dim_gen:alg} has $\Phi(\hat{A}) = 0$, where $\hat{A} = [\frac{x_{t'+2}}{\xi_0}  \cdots \frac{x_{t' + 2d}}{\xi_{d-1}}]$.  When $V$ is an $\eps$-net, $\Phi$ is a maximum of convex functions, and hence a convex function.
\end{remark}

\subsection{Epoch Notation.}
We define epochs in terms of rounds of system identification.  In particular, for the $k$th epoch $\tstart{k}$ is $t$ on the $k$th call to Alg.~\ref{sysid_multi_dim_gen:alg}.  and $\tend{k} = \min(\tstart{k+1} -1, T)$.  As such, within an epoch, $q$ is fixed, so we denote $q_k = \|x_{1:\tstart{k}}\|_2$ the value of $q$ within epoch $k$. Correspondingly, we denote the value of $q'$ in the $k$th epoch as $q_k'$.

\subsection{Estimation of the Control Matrix}\label{sec:B_est}
\begin{lemma}\label{control_matrix_estimation_gen:lem}
Suppose $\|f_{0:T-1}\|_2 \le q_k$ and $\alpha \ge 4^{2d}M^{2d}\eps^{-d}$, then in Alg.~\ref{controlid_multi_dim_gen:alg}, we have $\|x_{1:\tstart{k}+i}\|_2 \le  4^{2i}M^{2i}q_k\eps^{-i}$, for $0\le i \le d$.
\end{lemma}
\begin{proof}
We prove the lemma by induction. Note that if the lemma was true, no new epoch will start because $\|x_{1:t+i}\|_2 > \alpha q$ for any $i$. Now for the base case, note that for $i=0$, the inequality holds trivially. Suppose the condition holds for $i$. For $i+1$, we have
\begin{align*}
    \|x_{\tstart{k}+i+1}\|_2 &= \|A x_{\tstart{k}+i} + B u_{\tstart{k}+i} +z_{\tstart{k}+i}\|_2 \\
    &\le M\|x_{\tstart{k}+i}\|_2 + M\lambda_i + h\|x_{1:\tstart{k}+i}\|_2 + q_k\\
    &\le 4^{2i}M^{2i+1}q_k\eps^{-i} + 4^{2i}M^{2i+2}q_k\eps^{-(i+1)} + h4^{2i}M^{2i}q_k\eps^{-i} + q_k\\
    &\le 4^{2i+1}M^{2i+2}q_k\eps^{-(i+1)}
\end{align*}
Adding previous iterations, we have
\begin{align*}
\|x_{1:\tstart{k}+i+1}\|_2 &\le 4^{2i+1}M^{2i+2}q_k\eps^{-(i+1)} + 4^{2i}M^{2i}q_k\eps^{-i}\le 4^{2(i+1)}M^{2(i+1)}q_k\eps^{-(i+1)}.
\end{align*}
\end{proof}

\begin{lemma}\label{controlid_gen:lem}
Suppose $\|f_{0:T-1}\|_2 \le q_k$ and $\alpha \ge 4^{2d}M^{2d}\eps^{-d}$, then running Alg.~\ref{controlid_multi_dim_gen:alg} with $\eps \le \frac{L}{12\sqrt{d}}$ produces $\hat{B}$ such that $\|\hat{B}- B\|_2 \le 3 \eps\sqrt{d}$ and $\|B\hat{B}^{-1} - I\|_2 \le \frac{6\eps\sqrt{d}}{L} \le \frac{1}{2}$, with $\|x_{1:\tstart{k}+d}\|_2 \le   4^{2d}M^{2d}q_k\eps^{-d}$.
\end{lemma}
\begin{proof}
First note that as in Lem.~\ref{control_matrix_estimation_gen:lem}, no new epoch will start because $\|x_{1:t+i}\|_2 > \alpha q$ for any $i$. Let $i\in[0, d)$. Consider the estimation error of the $i+1$-th column of $B$:
\begin{align*}
    \|\frac{x_{\tstart{k}+i+1}}{\lambda_i} - Be_{i+1}\|_2 &= \frac{1}{\lambda_i}\|Ax_{\tstart{k}+i}+z_{\tstart{k}+i}\|_2 \le \frac{M}{\lambda_i}\|x_{\tstart{k}+i}\|_2 +\frac{1}{\lambda_i}\|z_{\tstart{k}+i}\|_2.
\end{align*}
By Lem.~\ref{control_matrix_estimation_gen:lem}, we have
$
\|x_{\tstart{k}+i}\|_2, \|w_{\tstart{k}+i}\|_2 \le 4^{2i}M^{2i}q_k\eps^{-i}.
$
Therefore we have $$
    \|\frac{x_{\tstart{k}+i+1}}{\lambda_i} - Be_{i+1}\|_2 \le \frac{M}{\lambda_i}\|x_{\tstart{k}+i}\|_2 +\frac{1}{\lambda_i}\|z_{\tstart{k}+i}\|_2 \le 3\eps.
$$
Concatenating the column estimates, we upper bound the Frobenius norm of $B - \hat{B}$,
\begin{align*}
    \|B - \hat{B}\|_F^2 &= \sum_{i=0}^{d-1} \|\frac{x_{\tstart{k}+i+1}}{\lambda_i} - Be_{i+1}\|_2^2 \le 9d\eps^2.
\end{align*}
We conclude that $\|B - \hat{B}\|_2\le \|B - \hat{B}\|_F \le 3\eps\sqrt{d}$. Moreover, with our choice of $\eps$, we have $\|B - \hat{B}\|_2\le \frac{L}{4}$, so by Ky Fan singular value inequalities, we have $\sigma_{\min}(B) \le  \sigma_{\min}(\hat{B}) + \frac{L}{4}$,
 and hence $\sigma_{\min}(\hat{B}) \ge \frac{L}{2}$, and the condition in Line 10 will not be triggered.

Now, we can write $B = \hat{B} + 3\eps\sqrt{d} C$ for some $C\in \reals^{d \times d}$, $\|C\| \le 1$. Then we have
\begin{align*}
      \|B\hat{B}^{-1} - I\| =3\eps\sqrt{d}\|C\hat{B}^{-1}\| \le \frac{3\sqrt{d}\eps}{\sigma_{\min}(\hat{B})} \le \frac{6\eps\sqrt{d}}{L}.  
\end{align*}
\end{proof}

\subsection{Estimation of the System}\label{sec:A_est}
\begin{lemma}\label{K_stability_gen:lem}
 Suppose $\|f_{0:T-1}\|_2 \le q_k$, $\alpha \ge 4^{2d}M^{2d}\eps^{-d}$, and Alg.~\ref{sysid_multi_dim_gen:alg} produces $\hat{A}$  such that $\|A- \hat{A}\|_2 \le \eps_A$ then the resultant controller $K$ satisfies $\|A- BK\| \le \eps_A + \frac{6\eps M\sqrt{d}}{L}$.
\end{lemma}
\begin{proof}
By  Lem.~\ref{controlid_gen:lem}, the algorithm will not start a new epoch with the choice of $\alpha$, and we have $\|B\hat{B}^{-1} - I\|_2  \le \frac{6\eps \sqrt{d}}{L}$, so we have 
\begin{align*}
 BK = B\hat{B}^{-1}\hat{A} = \hat{A} + \frac{6\eps \sqrt{d}}{L} C\hat{A}   
\end{align*}
for $C$ with $\|C\|_2 \le 1$.  Thus, we have 
\begin{align*}
    \|A - BK\|_2  &\le \|A - \hat{A}\|_2 + \frac{6\eps \sqrt{d}}{L} \|C\|\|\hat{A}\| \le \eps_A + \frac{6\eps M\sqrt{d}}{L} ~.
\end{align*}
\end{proof}

\begin{lemma}\label{approx_A_gen:lem_2}
 Suppose $\|f_{0:T-1}\|_2 \le q_k$ and $\alpha > R= (4M)^{5N}\eps^{-2N}$, then Alg.~\ref{sysid_multi_dim_gen:alg} produces $\hat{A}$ such that
 \begin{align*}
 \max_{v \in V} \|(A- \hat{A}) v\|_2 \le \frac{28\eps M\sqrt{d}}{L} +  3h~,
 \end{align*}
 with $\|x_{1:t'+2N}\|_2 \le  Rq_k$.
\end{lemma}
\begin{proof}
We first note by choice of $\alpha$, the SysID will not be restarted. We first upper bound $\Phi(A)$.  We also have $\Phi(\hat{A}) \leq \Phi(A)$ by optimality of $\hat{A}$. 

Let $i\in[0, N)$. Consider the estimation error of $Av_i$:
\begin{align*}
    \|\frac{x_{t'+2i+2}}{\xi_i} - AB\hat{B}^{-1}v_{i}\|_2 &= \frac{1}{\xi_i}\|A^2x_{t'+2i}+Az_{t'+2i}+z_{t'+2i+1}\|_2\\
    &\le \frac{M^2}{\xi_i}\|x_{t'+2i}\|_2 + \frac{M}{\xi_i}\|z_{t'+2i}\|_2+\frac{1}{\xi_i}\|z_{t'+2i+1}\|_2~.
\end{align*}
By Lemma \ref{approx_A_gen:lem_1}, we have
$
\|x_{t'+2i}\|_2, \|w_{t'+2i}\|_2 \le 4^{3i}M^{3i}q_k'\eps^{-i}.
$
Therefore for the first two terms we have,
$$
\frac{M^2}{\xi_i}\|x_{t'+2i}\|_2 + \frac{M}{\xi_i}\|z_{t'+2i}\|_2 \le \frac{M^2}{\xi_i}\|x_{t'+2i}\|_2 + \frac{M}{\xi_i}(\|w_{t'+2i}\|_2+ \|f_{t'+2i}\|_2) \le 3\eps.
$$
For the trajectory-dependent noise at time $t'+2i+1$, we have
\begin{align*}
  \frac{1}{\xi_i}\|w_{t'+2i+1}\|_2  \le \frac{h}{\xi_i} \|x_{1:t'+2i+1}\|_2 &\le \frac{h}{\xi_i}(\|x_{1:t'+2i}\|_2+ \|x_{t'+2i+1}\|_2)\\
  &\le  \frac{h}{\xi_i}(4^{3i}M^{3i}q_k'\eps^{-i}+ \|Ax_{t'+2i}+\xi_i B\hat{B}^{-1}v_i + z_{t'+2i}\|_2)\\
  &\le h\eps + \frac{hM}{\xi_i}\|x_{t'+2i}\|_2 + h\|B\hat{B}^{-1}\| +\frac{h}{\xi_i}\|z_{t'+2i}\|_2\\
  &\le 4\eps + h\|B\hat{B}^{-1}\| \le 4\eps + h(1 + \frac{1}{2})~.
\end{align*}
The last inequality holds, via Lem.~\ref{controlid_gen:lem}.
Therefore we have
\begin{align*}
\|\frac{x_{t'+2i+2}}{\xi_i} - AB\hat{B}^{-1}v_i\|_2 &\le 8\eps+\frac{3h}{2}.
\end{align*}
Adding the error induced by the bias of $\hat{B}$, 
\begin{align*}
    \|\frac{x_{t'+2i+2}}{\xi_i} - Av_i\|_2 &\le \|\frac{x_{t'+2i+2}}{\xi_i} - AB\hat{B}^{-1}v_i\|_2 + \|AB\hat{B}^{-1}v_i - Av_i\|_2\\
    &\le 8\eps + \frac{3h}{2} + \|A\|\|B\hat{B}^{-1} - I\|\\
    &\le 8\eps + \frac{3h}{2} + \frac{6M\sqrt{d}\eps}{L} \le \frac{14\eps M\sqrt{d}}{L} + \frac{3h}{2}~.
\end{align*}

Therefore, we have $\Phi(\hat{A}) \le \Phi(A)\le \frac{14\eps M\sqrt{d}}{L} + \frac{3h}{2}$ and it follows that 
\begin{align*}
    \max_{v \in V} \|(A- \hat{A})v\|_2 &= \max_{i \in [0,N)}\|(A- \hat{A})v_i\|_2 \\
    &\le  \max_{i \in [0,N)} \Big(\|Av_{i} - \frac{x_{t'+2i+2}}{\xi_{i}}\|_2 + \|\hat{A}v_{i} - \frac{x_{t'+2i+2}}{\xi_{i}}\|_2\Big)\\
    &\le \Phi(A) + \Phi(\hat{A}) \le \frac{28\eps M\sqrt{d}}{L} +  3h~.
\end{align*}

Finally, for the state magnitude at the final iteration, by Lem.~\ref{approx_A_gen:lem_1}
\begin{align*}
    \|x_{t'+2N}\|_2 &= \|Ax_{t'+2N-1} + w_{t'+2N-1} + f_{t'+2N-1}\|_2\\
    &\le M\|x_{t'+2N-1}\|_2 + h\|x_{1:t'+2N-1}\|_2+q_k\\
    &\le 4^{3N-1}M^{3N}q'_k\eps^{-N} + h4^{3N-1}M^{3N-1}q_k'\eps^{-N}+q_k\\
    &\le 3\cdot4^{3N-1}M^{3N}q'_k\eps^{-N}.
\end{align*}
Adding previous iterations, we have
$
\|x_{1:t'+2N}\|_2 \le 4^{3N}M^{3N}q'_k\eps^{-N} \le 4^{5N}M^{5N}\eps^{-2N}q_k.
$
\end{proof}

\begin{lemma}\label{approx_A_gen:lem_1}
Suppose $\|f_{0:T-1}\|_2 \le q_k$ and $\alpha >R= (4M)^{5N}\eps^{-2N}$, then in Alg.~\ref{sysid_multi_dim_gen:alg}, for odd iterations after $t'$, we have $\|x_{1:t'+2i+1}\|_2 \le  4^{3i+2}M^{3i+2}q'_k\eps^{-(i+1)}$, and for even iterations we have $\|x_{1:t'+2i}\|_2 \le  4^{3i}M^{3i}q'_k\eps^{-i}$, for $0\le i < N$.
\end{lemma}

\begin{proof}
We prove this by induction.  Note, by the condition on $\alpha$, SysID will not be restarted as long as our bounds on $\|x_{1:t'+ j}\|_2$ hold. For the base case, note that for $i=0$, the even case holds because by Lemma \ref{control_matrix_estimation_gen:lem}, $\|x_{1:\tstart{k}+d}\|_2 \le q'_k$. For the odd case, we have
\begin{align*}
    \|x_{t'+1}\|_2 &\le \|Ax_{t'}\|_2 + \xi_0\|B\hat{B}^{-1}\|_2 + \|z_{t'}\|_2\\
    &\le M\|x_{t'}\|_2 + 2\xi_0 + hq'_k + q_k \\
    &\le 3Mq'_k + \frac{2M^2q'_k}{\eps} \le\frac{5M^2q'_k}{\eps},
\end{align*}
where the first inequality holds by Lemma \ref{controlid_gen:lem}.
Adding the previous iterations, we have $\|x_{1:t'+1}\|_2\le 6M^2q'_k\eps^{-1} \le 4^{2}M^{2}q'_k\eps^{-1}$. Now, suppose the conditions hold for both even and odd iterations for $i$. For $i+1$, for the even iteration,
\begin{align*}
    \|x_{t'+2(i+1)}\|_2 &= \|Ax_{t'+2i+1} + w_{t'+2i+1} + f_{t'+2i+1}\|_2\\
    &\le M\|x_{t'+2i+1}\|_2 + h\|x_{1:t'+2i+1}\|_2+q_k\\
    &\le 4^{3i+2}M^{3(i+1)}q'_k\eps^{-(i+1)} + h4^{3i+2}M^{3i+2}q'_k\eps^{-(i+1)}+q_k\\
    &\le 3\cdot4^{3i+2}M^{3(i+1)}q'_k\eps^{-(i+1)}.
\end{align*}
Adding previous iterations, we have
$$
\|x_{1:t'+2(i+1)}\|_2 \le 4^{3(i+1)}M^{3(i+1)}q'_k\eps^{-(i+1)}.
$$
For the odd iteration, 
\begin{align*}
    \|x_{t'+2(i+1)+1}\|_2 &= \|Ax_{t'+2(i+1)} + Bu_{t'+2(i+1)} +w_{t'+2(i+1)} + f_{t'+2(i+1)}\|_2\\
    &\le M\|x_{t'+2(i+1)}\|_2 + 2\xi_{i+1}+ h\|x_{1:t'+2(i+1)}\|_2+q_k\\
    &\le 4^{3i+3}M^{3i+4}q'_k\eps^{-(i+1)} + 2\cdot4^{3i+3}M^{3i+5}q'_k\eps^{-(i+2)} + h4^{3i+3}M^{3i+3}q_k'\eps^{-(i+1)}+q_k\\
    &\le 5\cdot4^{3i+3}M^{3i+5}q'_k\eps^{-(i+2)}.
\end{align*}
Adding the previous iterations, we have
\begin{align*}
  \|x_{1:t'+2(i+1)+1}\|_2 \le 4^{3(i+1)+2}M^{3(i+1)+2}q'_k\eps^{-(i+2)} ~.  
\end{align*}
\end{proof}

\subsection{Cost of linear control}\label{sec:exploitation}
\begin{lemma}\label{stab_state_bound_multi_dim_gen:lem}
If $\|f_{0: T-1}\|_2 \le q_k$, and $u_{t} = -K x_{t}$ for $t \ge t^* \ge \tstart{k}$, with $\|A-BK\|_2 \le 1/2$ then for $h \le \frac{1}{6}$,  
\begin{align*}
    \|x_{1:\tend{k}}\|^2_2 \le \frac{ 18 \|x_{1:t^{*}}\|^2_2 +  72 q^2_k}{7}.
\end{align*}
\end{lemma}
\begin{proof}
We first prove that $\|x_{t^{*}: t}\|^2_2 \le 4\|z_{t^{*}: t-1}\|^2_2 + 2\|x_{t^{*}}\|^2_2$ by induction on $t \geq t^{*}$.  For the base case, we have $\|x_{t^{*}: t^{*}}\|^2_2 \le 2 \|x_{t^{*}}\|^2_2$.  Now note that
\begin{align*}
   \|x_{t^{*}: t+1}\|^2_2 = \sum_{s=t^{*}}^{t+1} \|x_s\|^2_2 &= \|x_{t^{*}}\|^2_2 + \sum_{s=t^{*}}^{t} \|x_{s+1}\|^2_2\\
   &= \|x_{t^{*}}\|^2_2 + \sum_{s=t^{*}}^{t} \|(A - BK)x_s + z_s\|^2_2\\
   &\le \|x_{t^{*}}\|^2_2 + 2\sum_{s=t^{*}}^{t} \|(A - BK)\|^2_2\|x_s\|^2_2 + \|z_s\|^2_2\\
   &\le \|x_{t^{*}}\|^2_2 + \frac{1}{2}\sum_{s=t^{*}}^{t} \|x_s\|^2_2 + 2\sum_{s=t^{*}}^{t} \|z_s\|^2_2\\
   &= \|x_{t^{*}}\|^2_2 + \frac{\|x_{t^{*}: t}\|^2_2}{2} + 2\|z_{t^{*}: t}\|^2_2.
\end{align*}

Applying the inductive hypothesis, we have 
\begin{align*}
\|x_{t^{*}:t+1}\|^2_2 \le \|x_{t^{*}}\|^2_2 + \frac{\|x_{t^{*}:t}\|^2_2}{2} + 2\|z_{t^{*}:t}\|^2_2
    &\le \|x_{t^{*}}\|^2_2 + \frac{4\|z_{t^{*}:t-1}\|^2_2 + 2\|x_{t^{*}}\|^2_2}{2} + 2\|z_{t^{*}:t}\|^2_2 \\
    &\le 2\|x_{t^{*}}\|^2_2 + \frac{4 \|z_{t^{*}:t}\|^2_2}{2} + 2\|z_{t^{*}:t}\|^2_2\\
    &\le 4\|z_{t^{*}:t}\|^2_2 + 2\|x_{t^{*}}\|^2_2.
\end{align*}

Adding $\|x_{1:t^{*} -1}\|^2_2$ to both sides, we have, for $t\ge t^{*}$, $\|x_{1:t}\|^2_2 \le 2 \|x_{1:t^{*}}\|^2_2 + 4\|z_{t^{*}:t-1}\|^2_2$.

Using Assumption~\ref{assumption:noise_bound} and using the shorthand $w_s$ for $w_s(x_{1:s})$, we have 
\begin{align*}
   \|z_{t^{*}:t-1}\|^2_2 &= \sum_{s=t^{*}}^{t-1} \|z_s\|^2_2 =\sum_{s=t^{*}}^{t-1} \|w_s + f_s\|^2_2  \\
   &\le 2\sum_{s=t^{*}}^{t-1} \|w_s\|^2_2 + \|f_s\|^2_2 \\
   &\le 2h^2\|x_{1:t-1}\|^2_2 + 2\|f_{0:t-1}\|^2_2.
\end{align*}
Using this bound, we have
\begin{align*}
    \|x_{1:t}\|^2_2 \le 2 \|x_{1:t^{*}}\|^2_2 + 8h^2 \|x_{1:t-1}\|^2_2 + 8 \|f_{0:t-1}\|^2_2 \le  2 \|x_{1:t^{*}}\|^2_2 + 8h^2 \|x_{1:t}\|^2_2 + 8 \|f_{0:t-1}\|^2_2~.
\end{align*}
Rearranging and bounding using $h =\frac{1}{6}$, we have
\begin{align*}
    \|x_{1:t}\|^2_2 \le \frac{ 2 \|x_{1:t^{*}}\|^2_2 +  8 \|f_{0:t-1}\|^2_2}{1-8h^2} \le \frac{18\|x_{1:t^{*}}\|^2_2 +  72 \|f_{0:t-1}\|^2_2}{7}~.
\end{align*}
The result follows using $t= \tend{k}$ and using $\|f_{0:T-1}\| \le q_k$.
\end{proof}

\subsection{Exploration on Standard Basis}
We consider the case where $V = \{e_1, e_2, \dots, e_d\}$. 

\begin{lemma}\label{std_basis_epoch_cost:lem}
Suppose $h \le \frac{1}{12\sqrt{d}}$,  $V = \{e_1, e_2, \dots, e_d\}$, and $\eps = \frac{L}{150Md}$, then if  $\|f_{0:T-1}\|_2 \le q_k$ and $\alpha = \big(\frac{4^{14} M^{8}d^{2}}{L^{2}}\big)^d$, the running Alg.~\ref{l2_gain_multi_dim_gen:alg} has states bounded by
\begin{align*}
    \|x_{1:\tend{k}}\|_2 \le \alpha q_k~.
\end{align*}
\end{lemma}
\begin{proof}
We first note that $\alpha$ is sufficiently large such that Lem.~\ref{approx_A_gen:lem_2} holds, and we have for each $i$, $\|(A-\hat{A})e_i\|_2 \le \frac{28\eps M\sqrt{d}}{L} +  3h$.  Now we note,

\begin{align*}
    \|A - \hat{A}\|_2 \le \|A - \hat{A}\|_F = \sqrt{\sum_{i=1}^d \|Ae_i - \hat{A} e_i\|_2^2} \le \sqrt{d}\big(\frac{28\eps M\sqrt{d}}{L} +  3h\big).
\end{align*}
Applying, Lem.~\ref{K_stability_gen:lem} and plugging in bounds on $\eps$ and $h$, we have 
\begin{align*}
    \|A- BK\|_2 \le  \|A - \hat{A}\|_2 + \frac{6\eps M\sqrt{d}}{L} \le \frac{34\eps Md}{L} +  3h\sqrt{d} \le \frac{34}{150} + \frac{1}{4} < \frac{1}{2}~.
\end{align*}
Now applying, Lem.~\ref{stab_state_bound_multi_dim_gen:lem} along with the state bound $\|x_{1:t'+2d}\| \le (4M)^{5d} \eps^{-2d}q_k$ from Lem.~\ref{approx_A_gen:lem_2}, we have 
\begin{align*}
    \|x_{1:\tend{k}}\|^2_2 \le \frac{ 18 ((4M)^{5d} \eps^{-2d}q_k)^2 +  72 q^2_k}{7} \le ((4M)^{6d} \eps^{-2d}q_k)^2 ~.
\end{align*}
Noting that $\eps > \frac{L}{4^4 Md}$, we get our result by bounding $(4M)^{6d} \eps^{-2d}$. 
\end{proof}
\subsection{Exploration on \texorpdfstring{$\eps$}--net}
We consider the case where $V$ is an $\eps$-net of the unit sphere.

From Lemma 5.3 of \cite{vershynin2011introduction}, there exists an $\eps$-net for the unit sphere of size $\big(1 + \frac{2}{\eps}\big)^d$.  We consider $V = \mathcal{N}_{1/2,d}$ such that $N = |V| = 5^d$.

\begin{lemma}\label{eps_net_epoch_cost:lem}
Suppose $h \le \frac{1}{15}$,  $V = \mathcal{N}_{1/2,d}$, and $\eps = \frac{L}{1000M\sqrt{d}}$, then if  $\|f_{0:T-1}\|_2 \le q_k$ and $\alpha = (\frac{4^{16} M^8 d}{L^2})^{5^d}$, the running Alg.~\ref{l2_gain_multi_dim_gen:alg} has states bounded by
\begin{align*}
    \|x_{1:\tend{k}}\|_2 \le \alpha q_k~.
\end{align*}
\end{lemma}
\begin{proof}
By Lem.~\ref{approx_A_gen:lem_2}, we have for each $v\in \mathcal{N}_{1/2,d}$, $\|(A-\hat{A})v\|_2 \le \frac{28\eps M\sqrt{d}}{L} +  3h$.  Now, we note that $\|A-\hat{A}\|_2 \le (1-1/2)^{-1}\max_{v \in \mathcal{N}_{1/2,d}} \|(A- \hat{A})v\|_2$ by a triangle inequality argument (see Lemma 5.4 of \cite{vershynin2011introduction}), so we have
$\|A - \hat{A}\|_2 \le \frac{56\eps M\sqrt{d}}{L} +  6h$. Applying, Lem.~\ref{K_stability_gen:lem} and plugging in bounds on $\eps$ and $h$, we have 
\begin{align*}
    \|A- BK\|_2 \le  \|A - \hat{A}\|_2 + \frac{6\eps M\sqrt{d}}{L} \le \frac{62\eps M\sqrt{d}}{L} +  6h \le \frac{62}{1000} + \frac{2}{5} < \frac{1}{2}~.
\end{align*}

Now applying, Lem.~\ref{stab_state_bound_multi_dim_gen:lem} along with the state bound $\|x_{1:t'+2N}\| \le (4M)^{5N} \eps^{-2N}q_k$ from Lem.~\ref{approx_A_gen:lem_2} with $N = 5^d$, we have 

\begin{align*}
    \|x_{1:\tend{k}}\|^2_2 \le \frac{ 18 ((4M)^{5N} \eps^{-2N}q_k)^2 +  72 q^2_k}{7} \le ((4M)^{6N} \eps^{-2N}q_k)^2 ~.
\end{align*}
Noting that $\eps > \frac{L}{4^5 M\sqrt{d}}$, we get our result by bounding $(4M)^{6N} \eps^{-2N}$. 
\end{proof}

\subsection{Final \texorpdfstring{$\ell_2$} --gain bounds}\label{final_bound:sec}
\begin{theorem}\label{std_basis:thm}
Suppose $h \le \frac{1}{12\sqrt{d}}$,  $V = \{e_1, e_2, \dots, e_d\}$, and $\eps = \frac{L}{150Md}$ and $\alpha = \big(\frac{4^{14} M^{8}d^{2}}{L^{2}}\big)^d$, then Alg.~\ref{l2_gain_multi_dim_gen:alg} has $\ell_2$ gain bounded by $\frac{10M^2\alpha^2}{L} < \big(\frac{4^{15} M^{10}d^{2}}{L^{3}}\big)^{2d}$.
\end{theorem}
\begin{proof}
First, observe that $\|x_{1:T}\|_2 \le \alpha q_k$, where $k$ is the final epoch. Indeed, if $\tstart{k} < T$, then by the design of the algorithm this condition is satisfied. Otherwise, we take $q_k = \|x_{1:T}\|_2$, and the algorithm stops before entering the system identification subroutine. Now we will show that running  Alg.~\ref{l2_gain_multi_dim_gen:alg}, $\|f_{0:T-1}\|_2 \ge \frac{q_k L}{10 M^2 \alpha}$, so 
\begin{align*}
    \frac{\|x_{1:T}\|_2}{\|f_{0:T-1}\|_2} \le \frac{10M^2\alpha^2}{L}~.
\end{align*}

We break into three cases:
\begin{enumerate}
    \item No failure occurred.
    \item $\sigma_{\min}(\hat{B}) < \frac{L}{2}$ in Alg.~\ref{controlid_multi_dim_gen:alg} (line 10).
    \item Failure check $\|x_{1:\tstart{k}}\|_2 > \alpha q_{k-1}$ occurs in Alg.~\ref{l2_gain_multi_dim_gen:alg} (line 5), Alg.~\ref{controlid_multi_dim_gen:alg} (line 4), or Alg.~\ref{sysid_multi_dim_gen:alg} (line 7),or 
    Alg.~\ref{controlid_multi_dim_gen:alg} (line 10).
\end{enumerate}

We first note that $q_k = \|x_{1:\tstart{k}}\|_2$ by definition.  We also note that if $k > 1$ (Cases $2$ and $3$), $\|f_{0:T-1}\|_2 > q_{k-1}$.  Suppose $\|f_{0:T-1}\|_2 \le q_{k-1}$, then by Lem.~\ref{std_basis_epoch_cost:lem} and choice of $\alpha$, the epoch $k-1$ would never have ended.  We now analyze each case separately.

\paragraph{Case 1: Failure never occurs}
Here we must have $\|x_{1:T}\|_2 = 0$ because $q$ is initialized at $0$.  $K$ is initialized to $0$, so  $\|u_{1:{T-1}}\|_2 = 0$ and  $\|f_{0:{T-1}}\|_2 = 0 = q$.

\paragraph{Case 2: Failure occurs in Alg. \ref{controlid_multi_dim_gen:alg} (line 10) second condition\\}
 We know $\sigma_{\min}(B) > L$, so we must have $\|\hat{B} - B\| > \frac{L}{2}$. By Lemma \ref{controlid_gen:lem}, if $\|f_{0:T-1}\| \le q_{k-1}$,  $\|\hat{B} - B\| \le 3 \sqrt{d} \eps \le \frac{L}{2}$, so by contradiction we must have $\|f_{0:T-1}\| > q_{k-1}$.  We now note that $q_k \le \alpha q_{k-1}$, otherwise, we would have failed the other condition of the if-statement. Combining, we have $\|f_{0:T-1}\| > \frac{q_k}{\alpha}$.

 \paragraph{Case 3: Failure occurs in Alg.~\ref{l2_gain_multi_dim_gen:alg} (line 5), in Alg.~\ref{controlid_multi_dim_gen:alg} (line 4), or Alg.~\ref{sysid_multi_dim_gen:alg} (line 7), or the first condition of Alg.~\ref{controlid_multi_dim_gen:alg} (line 10)\\}
 There are three possibilities for the control in the previous iteration: $u_{\tstart{k}-1} = -K x_{\tstart{k}-1}$, $u_{\tstart{k}-1} = 0$, or $u_{\tstart{k}-1}$ is from Alg.~\ref{controlid_multi_dim_gen:alg} (line 5) or Alg.~\ref{sysid_multi_dim_gen:alg} (line 11) and is a fixed control such that $\|u_{\tstart{k}-1}\|_2 < \alpha q_{k-1}$. To see this, we note that exploration controls are progressively increasing so we just need to look at the last large control played by Alg.~\ref{sysid_multi_dim_gen:alg}.  Thus, $\|u_{\tstart{k}-1}\|_2 \le \|\hat{B}^{-1}\| \xi_N \le \frac{2\xi_N}{L} \le \alpha q_{k-1}$.

For the first case, we note that
\begin{align*}
   \|K\|_2 = \|\hat{B}^{-1}\hat{A}\|_2 \le \|\hat{B}^{-1}\|_2\|\hat{A}\|_2 
    \le \frac{2M}{L} ~.
\end{align*}
Above, we use the fact that Alg.~\ref{controlid_multi_dim_gen:alg} always produces a $\hat{B}$ with $\sigma_{\min}(\hat{B}) \ge \frac{L}{2}$ and  $\|\hat{A}\|_2 < M$.  Noting that $\|x_{1:\tstart{k}-1}\|_2 \le \alpha q_{k-1}$, because otherwise the epoch would have ended on the previous iteration , we have $\|u_{\tstart{k}-1}\|_2 \le \frac{2M\alpha q_{k-1}}{L}$ in all cases.

  We now bound $\|x_{\tstart{k}}\|_2$ by applying the triangle inequality and system bounds:
\begin{align*}
    \|x_{\tstart{k}}\|_2 &= \|Ax_{\tstart{k}-1} + Bu_{\tstart{k}-1} +  w_{\tstart{k}-1} + f_{\tstart{k}-1}\|_2 \\
    &\le M\|x_{\tstart{k}-1} \|_2 + M \|u_{\tstart{k}-1}\|_2 + \|w_{\tstart{k}-1}\|_2 + \|f_{\tstart{k}-1}\|_2 \\
    &\le  M\| x_{1:\tstart{k}-1} \|_2 + M \|u_{\tstart{k}-1}\|_2 + h\|x_{1:\tstart{k}-1}\|_2 + \|f_{\tstart{k}-1}\|_2 \\
    & \le \frac{4M^2\alpha}{L} q_{k-1}  + \|f_{\tstart{k}-1}\|_2
\end{align*}

Adding the previous iterations, we have 
\begin{align*}
    q_k &= \|x_{1:\tstart{k}}\|_2 \le \|x_{1:\tstart{k}-1}\|_2 + \frac{4M^2\alpha}{L} q_{k-1}  + \|f_{\tstart{k}-1}\|_2 \\
    &\le \alpha q_{k-1} +\frac{4M^2\alpha}{L} q_{k-1}  + \|f_{\tstart{k}-1}\|_2 
    \le  \frac{5M^2\alpha}{L} q_{k-1}  + \|f_{\tstart{k}-1}\|_2~.
\end{align*}

Suppose $\|f_{\tstart{k}-1}\|_2> \frac{5M^2\alpha}{L} q_{k-1} $, then we immediately have $\|f_{0:T-1}\|_2 \ge \frac{q_k}{2}$.  Alternatively, we have $q_k \le \frac{10 M^2\alpha q_{k-1}}{L}$. Now since $\|f_{0:T-1}\|_2 > q_{k-1}$ , we have $\|f_{0:T-1}\|_2 > \frac{Lq_k}{10 M^2\alpha}$.

\end{proof}

\begin{theorem}\label{eps_net:them}
Suppose $h \le \frac{1}{15}$, $V = \mathcal{N}_{1/2,d}$, $\eps = \frac{L}{1000M\sqrt{d}}$, and $\alpha = (\frac{4^{16} M^8 d}{L^2})^{5^d}$, then Alg.~\ref{l2_gain_multi_dim_gen:alg} has $\ell_2$ gain bounded by $(\frac{4^{17} M^{10} d}{L^3})^{2 \cdot 5^d}$.
\end{theorem}
\begin{proof}
This follows exactly as Theorem~\ref{std_basis:thm} using Lem.~\ref{eps_net_epoch_cost:lem} in place of Lem.~\ref{std_basis_epoch_cost:lem}.

\end{proof}
\section{Cusumano-Poolla Algorithm}\label{poola:sec}
While Thm.~\ref{std_basis_abs:thm} and Thm.~\ref{eps_net:them} achieve robustness independent of any system parameters with exponential and doubly-exponential $\ltgain$ respectively, these algorithms are only applicable in the limiting settings where the control input matrix $B$ is full rank.  In contrast, the general task of designing adaptive controllers with finite closed loop $\ltgain$ has been solved by \citet{CusumanoPoollaACC1988}.  The Cusumano-Poolla algorithm works by iteratively testing all controllers, switching to a new controller when the gain exceeds some proposed bound.  As long as their is a candidate controller that can satisfy the proposed bound, this algorithm will eventually converge.  For completeness, we will analyze this algorithm, in the case that that our nominal 
linear dynamical system is strongly stabilizable  \citep{cohen2018online}, a quantitative notion of stabilizability.  Unlike our main setting, here we assume $B \in \R^{d \times p}$ so controls $u_t \in \R^p$.

\begin{definition}\label{strong_stabilizability:def}
A LDS $(A, B)$ is $\kappa$-strongly stabilizable if there exists a linear controller $K$ with $\|K\|_2\le \kappa$ such that $A +BK = HLH^{-1}$, where $H, L$ are matrices satisfying $\|H\|_2,\|H^{-1}\|_2, \|H\|_2\|H^{-1}\|_2\le \kappa$ and $\|L\|_2 \le 1 - \frac{1}{\kappa}$.
\end{definition}

We also will need to use an $\eps$-net over candidate controllers for this algorithm.

\begin{definition}\label{def:spec_eps_net}
Let $\K = \{K \in \R^{d\times p} : \|K\|_2 \le \kappa\}$ be the spectral norm ball of radius $\kappa$.
We define $\mathcal{N}^{*}_{\eps,\mK} \subseteq \mK$ to be an $\eps$-net of $\mK$,  with the metric $d(X,Y) = \|X-Y\|_2$ if for any $K \in \mK$, we have $K' \in \mathcal{N}^{*}_{\eps,\kappa, d \times p}$ such that $\|K-K'\|_2 \le \eps$.
\end{definition}

\begin{algorithm2e}[h]
\label{cusumano_poola:alg}
\KwIn{System upper bound M, strong controllability parameter $\kappa$, disturbance bound $F$.} 
Set $\alpha=27\kappa^4, \mK' = \mathcal{N}^{*}_{\frac{1}{2M\kappa^2},\mK}$ \\
Initialize $q=F$, pick arbitrary controller $K$ from $\mK'$.\\
\For{$t = 1 \ldots  T$}{
Observe $x_t$.\\
\eIf{ $\|x_{1:t}\|_2 > \alpha q$}{
Update $q = \|x_{1:t}\|_2$.\\
Update candidate controllers: $\mK' = \mK' \setminus \{K\}$.\\
Pick any $K$ from $\mK$.
}
{
Execute $u_t = -K x_t$.
}}
\caption{Cusumano-Poolla algorithm}
\end{algorithm2e}

\begin{theorem}
If $F= \|f_{0:T-1}\|_2$, and $h < \frac{1}{5\kappa^4}$, then Alg.~\ref{cusumano_poola:alg} has $\ltgain$ bounded by 
$ \big(135\kappa^5 M \big )^{(4M\kappa^3\sqrt{dp})^{dp}}$.
\end{theorem}
\begin{proof}
We first note that there exists a controller $K \in \K'$ such that Lem.~\ref{poolla_good_controller:lem} holds, in which case the controller will never switch. To get an $\ltgain$ bound, we need to bound the state expansion that occurs using any other $K \in \mK'$.  We note that $\|A+BK\|_2 \le 2\kappa M$, when a controller pushes $\|x_{1:t}\|_2$ above $\alpha q$, we have 
\begin{align*}
    \|x_{t}\|_2 \le 2 \kappa M \|x_{t-1}\|_2 + \|w_{t-1}\|_2 + \|f_{t-1}\| \le (2 \kappa \alpha M + h + 1) q ~.
\end{align*}

Adding in $\|x_{1:t-1}\|_2 < \alpha q$, we have $\|x_{1:t}\|_2 \le 5\kappa \alpha M q$.  This state expansion can occur at most once per controller in $\K'$. By Lem.~\ref{spec_eps_net:lem}, $|\K'| < (4M\kappa^3\sqrt{dp})^{dp}$ and hence,
\begin{align*}
    \|x_{1:T}\|_2 \le \big(5\kappa \alpha M \big )^{(2M\kappa^3\sqrt{dp})^{dp}} F~.
\end{align*}
Plugging in $\alpha$, yields the result.
\end{proof}

\begin{lemma}\label{spec_eps_net:lem}
 There exists an $\eps$-net for $\mK$ such that  $|\mathcal{N}^{*}_{\eps,\mK}| \le \big ( \frac{2\kappa \sqrt{dp}}{\eps}\big)^{dp}$.
\end{lemma}
\begin{proof}
This follows by choosing a grid with granularity $\frac{\eps}{\sqrt{pd}}$ in each coordinate, assuring that for any $K$ there is a $K'$ in the net at most $\frac{\eps}{\sqrt{pd}}$ away in each coordinate.  This guarantees that $\|K -K'\|_F \le \eps$ and so $\|K - K'\|_2 \le \eps$.
\end{proof}

\begin{lemma}\label{poolla_good_controller:lem}
If $(A, B)$ is $\kappa$-strongly stabilizable, and $h < \frac{1}{5\kappa^{4}}$ then there exists a linear controller $K \in \mathcal{N}^{*}_{\frac{1}{2M\kappa^2},\mK}$ such that if the controller is played from iteration $t^{*}$ to iteration $T$, the states are bounded by

\begin{align*}
    \|x_{1:T}\|_2 \le \sqrt{100 \kappa^4 \|x_{1:t^{*}}\|^2_2 + 625\kappa^7 \|f_{0:T-1}\|^2_2} \le 27\kappa^{4} \max(\|x_{1:t^{*}}\|_2,\|f_{0:T-1}\|_2)~.
\end{align*}
\end{lemma}
\begin{proof}
From strong stablizability, for some $K^{*} \in \K$ we can write $A+BK^{*} = HLH^{-1}$ with $\|H\|_2\|H^{-1}\|_2 \le \kappa$ and $\|L\|_2 \le 1 - \frac{1}{\kappa}$. We define $y_t = H^{-1} x_t$.  Let $K$ be a controller in the $\eps$-net such that $\|K -K^{*}\| \le \frac{1}{M\kappa^2}$.

We first prove that $\|y_{t^{*}: t}\|^2_2 \le 12\kappa^5 \|z_{t^{*}: t-1}\|^2_2 + 4\kappa^2\|y_{t^{*}}\|^2_2$ by induction on $t \geq t^{*}$.  For the base case, we have $\|y_{t^{*}: t^{*}}\|^2_2 \le 4\kappa^2 \|y_{t^{*}}\|^2_2$ since $\kappa \geq 1$.  Now note that
\begin{align*}
   \|y_{t^{*}: t+1}\|^2_2 = \sum_{s=t^{*}}^{t+1} \|y_s\|^2_2 &= \|y_{t^{*}}\|^2_2 + \sum_{s=t^{*}}^{t} \|y_{s+1}\|^2_2\\
   &= \|y_{t^{*}}\|^2_2 + \sum_{s=t^{*}}^{t} \|H^{-1}((A - BK)x_s + z_s)\|^2_2\\
   &= \|y_{t^{*}}\|^2_2 + \sum_{s=t^{*}}^{t} \|H^{-1}((A - BK^{*} + B(K - K^{*})))x_s + z_s)\|^2_2\\
   &= \|y_{t^{*}}\|^2_2 + \sum_{s=t^{*}}^{t} \|(L + H^{-1}B(K - K^{*})H) y_t +  H^{-1}z_s\|^2_2\\
   &\le \|y_{t^{*}}\|^2_2 + (1+\frac{1}{2\kappa})\sum_{s=t^{*}}^{t} \|L + H^{-1}B(K - K^{*})H\|^2_2\|y_s\|^2_2 + (1 + 2\kappa)\sum_{s=t^{*}}^{t}\kappa^2\|z_s\|^2_2\\
   &\le \|y_{t^{*}}\|^2_2 + (1- \frac{1}{4\kappa^2})\sum_{s=t^{*}}^{t} \|y_s\|^2_2 + 3\kappa^3\sum_{s=t^{*}}^{t} \|z_s\|^2_2\\
   &= \|y_{t^{*}}\|^2_2 + (1- \frac{1}{4\kappa^2})\|y_{t^{*}: t}\|^2_2 + 3\kappa^3\|z_{t^{*}: t}\|^2_2~.
\end{align*}

Above, we use the fact that 
\begin{align*}
    \|L + H^{-1}B(K - K^{*})H\|_2 \le \|L\|_2 + \|H\|_2\|H^{-1}\|_2\|B\|_2\|K - K^{*}\|_2 \le (1 - \frac{1}{\kappa}) + \kappa M \cdot \frac{1}{2M\kappa^2} \le 1-\frac{1}{2\kappa}~.
\end{align*}

Applying the inductive hypothesis, we have 
\begin{align*}
\|y_{t^{*}:t+1}\|^2_2 &\le \|y_{t^{*}}\|^2_2 + (1- \frac{1}{4\kappa^2})\|y_{t^{*}: t}\|^2_2 + 3\kappa^3\|z_{t^{*}: t}\|^2_2\\
    &\le \|y_{t^{*}}\|^2_2 + (1- \frac{1}{4\kappa^2})(12\kappa^5 \|z_{t^{*}: t-1}\|^2_2 + 4\kappa^2\|y_{t^{*}}\|^2_2) + 3\kappa^3\|z_{t^{*}: t}\|^2_2 \\
    &\le 4\kappa^2\|y_{t^{*}}\|^2_2 + 12\kappa^5 \|z_{t^{*}: t}\|^2_2~.\\
\end{align*}

Adding $\|y_{1:t^{*} -1}\|^2_2$ to both sides, we have, for $t\ge t^{*}$, $\|y_{1:t}\|^2_2 \le 4\kappa^2 \|y_{1:t^{*}}\|^2_2 + 12\kappa^5\|z_{t^{*}:t-1}\|^2_2$.  Now, by definition of $y_t$ and $\|H\|_2\|H^{-1}\|_2 \leq \kappa$, we note that 
 $\|x_{1:t}\|^2_2 \le 4\kappa^4 \|x_{1:t^{*}}\|^2_2 + 12\kappa^7\|z_{t^{*}:t-1}\|^2_2$.

Using Assumption~\ref{assumption:noise_bound} and using the shorthand $w_s$ for $w_s(x_{1:s})$, we have 
\begin{align*}
   \|z_{t^{*}:t-1}\|^2_2 &= \sum_{s=t^{*}}^{t-1} \|z_s\|^2_2 =\sum_{s=t^{*}}^{t-1} \|w_s + f_s\|^2_2  \\
   &\le 2\sum_{s=t^{*}}^{t-1} \|w_s\|^2_2 + \|f_s\|^2_2 \\
   &\le 2h^2\|x_{1:t-1}\|^2_2 + 2\|f_{0:t-1}\|^2_2.
\end{align*}
Using this bound, we have
\begin{align*}
    \|x_{1:t}\|^2_2 \le 4\kappa^4 \|x_{1:t^{*}}\|^2_2 + 24\kappa^7h^2 \|x_{1:t-1}\|^2_2 + 24\kappa^7 \|f_{0:t-1}\|^2_2 \le  4\kappa^4 \|x_{1:t^{*}}\|^2_2 + 24\kappa^7h^2 \|x_{1:t}\|^2_2 + 24\kappa^7 \|f_{0:t-1}\|^2_2~.
\end{align*}
Rearranging and plugging in the bound on $h$, we have
\begin{align*}
    \|x_{1:t}\|^2_2 \le \frac{ 4\kappa^4 \|x_{1:t^{*}}\|^2_2 +  24\kappa^7 \|f_{0:t-1}\|^2_2}{1-24\kappa^7h^2} \le 100 \kappa^4 \|x_{1:t^{*}}\|^2_2 + 625\kappa^7 \|f_{0:t-1}\|^2_2~.
\end{align*}
\end{proof}

\begin{remark}
For simplicity, Alg.~\ref{cusumano_poola:alg} assumes we know the total disturbance magnitude $\|f_{1:T}\|_2$ exactly.  The same doubling scheme from Alg.~\ref{l2_gain_multi_dim_gen:alg} can be adapted if the disturbance magnitude is not known by looping through the full $\eps$-net of controllers in epochs until an appropriate $F$ is found.
\end{remark}

\end{document}